\documentclass[12pt]{article}

\usepackage{amssymb}
\usepackage{amsmath}
\usepackage{amsthm}
\usepackage{amscd}
\usepackage{enumerate}
\usepackage{color}
\usepackage[usenames,dvipsnames,svgnames,table]{xcolor}
\usepackage[margin=1.11in]{geometry}%
\usepackage{hyperref}
\usepackage[affil-it]{authblk}
\usepackage{tikz}
\usepackage{xcolor}
\usetikzlibrary{arrows,decorations.pathmorphing,backgrounds,fit,automata,shapes.geometric,shapes.arrows,calc}
\usetikzlibrary{positioning}

\theoremstyle{definition}
\newtheorem{defn}{Definition}[section]
\newtheorem{definition}[defn]{Definition}

\theoremstyle{plain}
\newtheorem{theorem}[defn]{Theorem}
\newtheorem{proposition}[defn]{Proposition}
\newtheorem{corollary}[defn]{Corollary}
\newtheorem{lemma}[defn]{Lemma}

\newtheorem{cor}[defn]{Corollary}

\newtheorem{lem}[defn]{Lemma}

\newtheorem{prop}[defn]{Proposition}

\newtheorem{thm}[defn]{Theorem}

\theoremstyle{remark}
\newtheorem{remark}{Remark}[section]

\newtheorem{examples}{Examples}[section]
\newtheorem{example}{Example}[section]

\DeclareMathOperator{\diam}{diam}

\newcommand{\set}[1]{\left\{ #1 \right\}}
\newcommand{\paren}[1]{\left( #1 \right)}
\newcommand{\gen}[1]{\left\langle #1 \right\rangle}

\newcommand{\norm}[1]{\left\| #1 \right\|}
\newcommand{\abs}[1]{\left| #1 \right|}
\newcommand{\N}{\mathbb N}
\newcommand{\R}{\mathbb R}

\newcommand{\A}{\mathcal{A}}

\newcommand{\round}[1]{{\ooalign{\hfil\raise .10ex\hbox{\scriptsize#1}\hfil\crcr\mathhexbox20D}}}

\newcommand{\cell}{\operatorname{G}}

\newcommand{\B}{\mathcal{B}}

\newcommand{\eng}{\mathcal{E}}
\newcommand{\Dom}{\operatorname{Dom}}
\newcommand{\dom}{\operatorname{Dom}}

\newcommand{\maxim}{\mathcal{M}}

\newcommand{\SG}{\operatorname{SG}}
\newcommand{\edge}{\mathsf{E}}

\newcommand{\spn}{\operatorname{span}}

\newcommand{\irt}{\mathcal{H}}

\newcommand{\diralg}{\mathcal{C}}
\newcommand{\z}{\mathbf z}

\newcommand{\var}{\operatorname{Var}}
\newcommand{\per}{\operatorname{Per}}
\newcommand{\h}{\mathfrak{h}}

\newcommand{\graph}{\mathsf{G}}
\newcommand{\gmet}{\graph^{\text{met}}}

\newcommand{\cdc}{(CdC)\ }
\newcommand{\BE}{(BE)\ }
\newcommand{\PI}{(PI)\ }
\newcommand{\VD}{(VD)\ }
\newcommand{\VLB}{(VLB)\ }
\newcommand{\LYI}{(LYE)\ }

\newcommand{\red}[1]{\textcolor{red}{#1}}

\title{Differential one-forms on Dirichlet spaces and Bakry-\'Emery estimates on metric graphs}
\author{Fabrice Baudoin%
  \thanks{Author supported in part by NSF Grant DMS 1660031}}
\affil{Department of Mathematics, University of Connecticut}

\author{Daniel J. Kelleher}
\affil{Department of Mathematics, Purdue University}

\begin{document}
\maketitle
\begin{abstract}
 We develop a general framework on Dirichlet spaces to prove a weak form of the Bakry--\'Emery estimate and study its consequences. This estimate may be satisfied in situations, like metric graphs, where generalized notions of Ricci curvature lower bounds are not available.
 \end{abstract}

%
\tableofcontents
\section{Introduction}

In the last few years, there has been much work toward defining curvature bounds for general metric measure  spaces (see  \cite{Amb, LV,St1,St2}). In this work, we are interested in mostly one-dimensional Dirichlet spaces, like metric graphs,  which are spaces for which good notions of curvature have been elusive so far. Despite the lack of good  \textit{curvature bounds} on those spaces, we prove that metric graphs with finite number of edges satisfy a weak form of the Bakry-\'Emery estimate:
\begin{align}\label{BE-intro}
\sqrt{\Gamma(e^{t \Delta} f)} \le C_1 e^{t \Delta} \sqrt{\Gamma(f)}, \quad 0 \le t \le 1,
\end{align}
where $\Delta$ is the generator of a Dirichlet form and $\Gamma$ the associated carr\'e du champ defined by
\[
\Gamma(f,g) = \frac12(\Delta(fg) - f\Delta g - g\Delta f).
\]
The equality defining $\Gamma$  is usually understood in the weak sense of \cite[Proposition 4.1.3]{BH91}. As with all bilinear operators we will  denote $\Gamma(f) = \Gamma(f,f)$. 

We show that for metric graphs the optimal $C_1$ in the inequality \eqref{BE-intro} is bounded from below by $(\max \deg v -1)$, where the maximum is taken over the set of vertices of the graph. Therefore, for non trivial graphs,  $C_1 >1$. Weak Bakry--\'Emery estimates of the type \eqref{BE-intro} with $C_1>1$, have already been met in the literature and are known to be satisfied in the H-type Heisenberg groups (see \cite{MR2240167,BBBC,MR2557945}). It turns out that weak Bakry--\'Emery estimates are actually sufficient to recover  several key consequences of the classical one (which corresponds to $C_1=1$), see \cite{BBBC} for heat kernel functional inequalities,  \cite{MR2320383} for the Hardy-Littlewood-Sobolev theory and \cite{Kuw10} for Wasserstein spaces Lipschitz continuity properties. Therefore, it is interesting to develop a general framework to understand them. 

\

In classical situations, the Bakry-\'Emery estimate with $C_1=1$ is proved as a consequence of a lower bound on the Bakry's $\Gamma_2$ operator
\[
\Gamma_2(f,g)=\frac{1}{2} ( \Delta \Gamma (f,g)-\Gamma(f,\Delta g)-\Gamma(g,\Delta f)).
\]
However, in a general framework, the  $\Gamma_2$ operator may not even be defined in a strong sense, since $\Gamma(f,g)$ fails to be in the domain of $\Delta$ for a reasonable class $\mathcal{C}$ of $f,g \in \mathcal{C}$. More recently, it has been proved in   \cite{AGS14} and \cite{AGS15}, that under mild conditions,  the Bakry-\'Emery estimate with $C_1=1$ is actually equivalent to the underlying metric space satisfying a Riemannian Ricci curvature lower bound in the sense of \cite{Amb}.

In singular spaces, where no generalized Ricci curvature lower bounds are satisfied and no $\Gamma_2$-calculus is available, to prove the weak Bakry-\'Emery estimate, it seems fruitful to take advantage of and further develop the theory of measurable one-forms on Dirichlet spaces that was originally devised in   \cite{CS03,Hin10, IRT12,HRT,HKT13}. A main observation is the intertwining  property
\begin{align}\label{PT-intro}
 \partial e^{t \Delta}=e^{t \vec{\Delta}}  \partial,
\end{align}
where $\partial$ is the exterior derivative and $e^{t \vec{\Delta}}$  a semigroup on \textit{one-forms}. Proving the semigroup domination
\begin{align}\label{PTR-intro}
\| e^{t \vec{\Delta}} \eta \| \le C_1 e^{t \Delta} \| \eta \|, \quad 0 \le t \le 1.
\end{align}
therefore implies \eqref{BE-intro}.

\

The paper is organized as follows. In Section 2, we present the necessary preliminaries about the space of measurable one-forms on a Dirichlet space and prove the intertwining \eqref{PT-intro}. As a consequence, we prove that the weak Bakry-\'Emery estimate is satisfied in large times, with exponential decay, in a general class of compact Dirichlet spaces. Admittedly, this is mainly a spectral effect. \textit{Curvature} is what controls the estimate in small times. 

\

In Section 3, we explore consequences of \eqref{BE-intro} that go beyond the usual applications. We are mostly  interested in the space of bounded variation functions and isoperimetric type inequalities. Bounded variation functions in Dirichlet spaces may be defined by adapting the original ideas of De Giorgi \cite{DG}. More precisely, one defines 
\[
\var f = \sup\set{ \gen{f,\partial^*  \eta }_2~|~  \eta  \in \dom \partial^*,  \| \eta \|   \leq 1}
\]
where $\partial^*$ is the adjoint of the exterior derivative (one may also  think of it as a generalized divergence in as distributional sense, see \cite{HRT} ).  If $f \in \dom \eng$, then $\var f$ is comparable to $\int \sqrt{\Gamma (f)} d\mu$, but in geometric measure theory, one is typically interested in situations where $f \notin \dom \eng$. With this definition in hand and the semigroup domination \eqref{PTR-intro}, we obtain Sobolev embeddings of the type
\begin{align*}
\| f \|_p \le C_Q (\var f +\|f\|_1)  ,
\end{align*}
where $p=\frac{Q}{Q-1}$ and $Q$ is the semigroup dimension (in the sense of Varopoulos). The corresponding isoperimetric inequality writes
\[
\mu(E)^{\frac{Q-1}{Q}} \le C_{\emph{iso}} P (E).
\]
where $P(E):=\var 1_E$ is the perimeter of a Caccioppoli set. We also prove generalizations of the Buser's and Ledoux's isoperimetric  inequalities. The main work is to adapt to our framework some ideas originally due to Ledoux \cite{BGL,Le, Le2}.

\

In Section 4, we specialize our study to the case of Hino index one Dirichlet spaces. In those spaces, one forms may be identified with functions. As a consequence  the semigroup domination \eqref{PTR-intro} is equivalent to a semigroup domination on functions:
\[
| e^{t \Delta^\perp} f | \le C e^{t \Delta} |f|
\]
where $\Delta^\perp$ is a self-adjoint operator (in general non-Markovian), that we call Poincar\'e dual of $\Delta$. For instance, if $\Delta$ is the Laplace operator on an interval with Neumann boundary conditions, then $\Delta^\perp$ is the Laplace operator with Dirichlet boundary conditions. We end the section with the study of the situation where the reference measure comes from a harmonic form. In this special situation, the Dirichlet space associated to $\Delta^\perp$ is an extension of the original Dirichlet space and the Bakry-\'Emery estimate with constant $C_1=1$ is satisfied for a suitable class of functions $f$. 

\

In the last Section 5, we study in detail a class of examples to which the previous results apply. We first show that on the Walsh spider with $N$ legs, one has

\begin{align*}
\sqrt{\Gamma(e^{t \Delta} f)} \le (N-1) e^{t \Delta} \sqrt{\Gamma(f)}, \quad t \ge 0.
\end{align*}

where the constant $N-1$ is optimal. The weak Bakry-\'Emery estimate is then generalized to any metric graph (with standard boundary conditions at the vertices) that has a finite number of edges. Some consequences are then explored.

\

\textbf{Acknowledgements}: The authors thank an anonymous referee for multiple remarks leading to a better presentation of this work and Alexander Teplyaev and Patricia Alonso Ruiz for stimulating discussions and pointing out a mistake in one of the proofs of an earlier version. 

\section{Differential one-forms on Dirichlet spaces}
\label{sec:SG}

\subsection{Preliminaries}

This section mostly establishes definitions and collects known results from the theory of differential forms on Dirichlet spaces as can be found in \cite{IRT12,HRT,HKT13}. These works built on ideas found in \cite{CS03}.   

\

Let $(\eng,\dom\eng)$ be a  symmetric strictly local regular  Dirichlet form on $L^2(X,\mu)$ with self-adjoint generator $\Delta$,  where $(X,d)$ is a locally compact separable metric space, and $\mu$ is a locally finite Borel measure such that $\mu (U) >0$ when $U$ is a non empty open set.  Throughout the paper it is always assumed that $\eng$ admits a carr\'e du champ (see  \cite{BH91}) which shall be denoted by $\Gamma$.  We adopt the convention that $\Delta $ is a negative operator in the sense that for $u \in \dom \Delta$,
\[
\eng(u,u)=-\int u \Delta u d\mu
\]


\begin{example}\label{example1}
If $\mathbb{M}$ is a smooth complete Riemannian manifold and 
\[
\mathcal{E}(f,g)=\int_\mathbb{M} \langle df , dg \rangle_{T^* \mathbb{M}} d\mu, \quad f,g \in W^{1,2}(\mathbb{M})
\]
where $\mu$ is the Riemannian volume measure, then $\Gamma (f,g)= \langle df , dg \rangle_{T^* \mathbb{M}}$ and the restriction of $\Delta$ to smooth functions coincides with the Laplace-Beltrami operator.
\end{example}
 
Let $\B_b(X)$ be defined to be the set of bounded Borel measurable functions on $X$, and $\diralg:= C_b(X)\cap \dom\eng$. By regularity, $\diralg$ is dense in $\dom \eng$. For simple tensors in the vector space $\diralg \otimes \B_b(X)$, define the scalar product
\[
\gen{ f_1\otimes g_1,f_2\otimes g_2}_\irt =\int g_1g_2 \Gamma(f_1,f_2) d\mu.
\]
The above product is non-negative definite, and thus by factoring out the 0-seminorm elements and completing defines a Hilbert space which will be denoted by $\irt$. We think of $\irt$ as the space of $L^2$ differential one-forms $\irt$ on $X$ (as in \cite{CS03}).

\begin{example}
For the Dirichlet form in Example \ref{example1}, one has
\[
\gen{ f_1\otimes g_1,f_2\otimes g_2}_\irt =\int_\mathbb{M} g_1g_2 \langle df_1 , df_2 \rangle_{T^* \mathbb{M}} d\mu.
\]
and, up to isomorphism, $\irt$ is the Hilbert space of square integrable one-forms (see \cite{HRT}).
\end{example}

The space $\irt$ is a $\diralg$-left module structure and $\B_b(X)$-right module with multiplication defined to be 
\[
a\cdot f\otimes g = (af)\otimes g - a\otimes (fg)\quad\text{and}\quad f\otimes g \cdot a = f\otimes ag
\]
respectively. Since  $\eng$ is strictly local,  left and right  multiplications  actually coincide (see \cite[Lemma 3.2.5]{FOT11}). We define an exterior derivative $\partial: \diralg \to \irt$ by 
\[
\partial f = f\otimes \mathbf 1,
\]
where $\mathbf 1$ is the constant function equal to 1. From the definition 
\[
\gen{\partial f,\partial g}_\irt = \int \Gamma(f,g) d\mu = \eng(f,g),
\]
and hence $\partial$ is a closable operator because $\eng$ is a closed Dirichlet form. Therefore $\partial$ extends to a densely defined closed linear operator $L^2(X,\mu) \to \mathcal{H}$ with domain
\[
\dom\partial=\dom\eng.
\]

\begin{example}
In the case of Example \ref{example1}, the restriction of $\partial$ to smooth functions coincides with the usual derivative $d$.
\end{example}

Note that since $\eng$ is  assumed to be strictly local $\partial$ admits a chain rule (see \cite[Lemma 3.2.5]{FOT11}): If $F$ is continuously differentiable and  $f^1,f^2,\ldots, f^n \in \dom\eng$, then
\[
\partial F(f^1,f^2,\ldots,f^n) = \sum_{i=1}^n\frac{\partial F}{\partial x^i}\partial f^i.
\]

The co-differential is defined as the adjoint of the exterior differential $\partial$. More precisely, the operator $\partial^*$ is the densely defined operator from $\irt \to L^2(X,\mu)$ with domain
\[
\dom\partial^* := \set{\eta \in \irt ~:~\exists f\in L^2(X,\mu),~\text{with } \gen{\eta,\partial \phi}_\irt = \gen{f,\phi}_2~\forall~\phi\in \dom \eng },
\]
and we have $\partial^* \eta = f$. The operator $\partial^*$ may also be interpreted in a distributional sense (see \cite{HRT}).

\begin{example}
In the case of Example \ref{example1}, the restriction of $\partial^*$ to smooth forms coincides with the usual  divergence  $\delta$.
\end{example}

Observe that $u\in \dom\eng$ is in $\dom\Delta$ if and only if there exists $\Delta u \in L^2(X,\mu)$ such that  $\eng(u,\phi) = -\gen{\Delta u,\phi}$ for all  $\phi\in\dom\eng$. Hence 
\[
\dom\Delta = \set{u\in\dom\eng~|~ \partial u \in \dom\partial^*}.
\]
and for $u\in\dom\Delta$ we have $\partial^*\partial u = -\Delta u$.

\subsection{Laplacian on one-forms}

Define the 1-form Laplacian, as in \cite{HKT13} by $\vec\Delta =- \partial\partial^*$ with the domain
\[
\dom\vec\Delta = \set{\eta \in \irt~|~\partial^* \eta \in \dom\partial}.
\]
Since $\partial^*$ is densely defined and closed, from a Von Neumann's theorem (confer \cite[theorem 8.4]{Tay} or the proof of theorem VIII.32 in \cite{RS72}),  the operator $\vec\Delta=- (\partial^*)^*\partial^*$ is self-adjoint. Alternatively, since $\partial^*$ is closed, we may also see $\vec\Delta$  as the self-adjoint generator of the closed symmetric bilinear form on $\mathcal{H}$
\[
\vec{\eng} (\omega, \eta)=\langle \partial^* \omega, \partial^* \eta \rangle_{L^2}.
\]
\begin{remark}
In the vein of Example \ref{example1}, if $X$ is a one-dimensional Riemannian manifold, the space of 2-forms is trivial, therefore the restriction of $\vec\Delta$ to smooth forms coincides with the Hodge--de Rham Laplacian $dd^* + d^*d$. In \cite{HT14}, it is shown more generally that for topologically one dimensional spaces, there are no associated differential 2-forms, and hence $\vec\Delta$ is analogous to the Hodge--de Rham Laplacian. This definition of form Laplacian is taylor made for  1-dimensional situations.
\end{remark}

Our first result is the following:

\begin{theorem}\label{intertwining1}
For $f \in \dom \eng$,
\[
  \partial e^{t\Delta} f=e^{t\vec\Delta} \partial f, \quad t \ge 0.
\]
\end{theorem}

\begin{proof}
For every $f \in \dom (\Delta), \omega \in \dom\partial^*$, one has
\[
\langle \Delta f , \partial^* \omega \rangle_2=-\vec{\eng} (\partial f, \omega).
\]
The result is then a consequence of Theorem 3.1 in  Shigekawa \cite{Shi00}.
\end{proof}

We can describe more precisely $\vec{\Delta}$ and its domain in the special case where $\Delta$ has pure point spectrum, i.e. there exists an increasing sequence $0\le \lambda_1 \le \lambda_2 \le \cdots$ of  eigenvalues of $-\Delta$, with finite multiplicity, and a complete orthonormal basis $(\phi_j)_{j\ge 1}$ of corresponding eigenfunctions such that
\[
-\Delta f =\sum_{j=1}^{+\infty} \lambda_j \langle f , \phi_j \rangle \phi_j, \quad f \in \dom\Delta.
\]

\begin{lemma}\label{spectre}
Assume that $\Delta$ has a pure point spectrum. Then, \[
\dom \partial^* =\left\{ \eta \in \mathcal{H}, \sum_{j=1}^{+\infty}  \langle \eta,\partial \phi_j \rangle_{\mathcal{H}}^2 <+\infty \right\},
\]
and for $\eta \in \dom \partial^*$,
\[
\partial^* \eta=\sum_{j=1}^{+\infty}  \langle \eta , \partial \phi_j \rangle_\irt  \phi_j.
\]
Furthermore
\[
\dom\vec\Delta =\left\{ \eta \in \mathcal{H}, \sum_{j=1}^{+\infty} \lambda_j \langle \eta , \partial \phi_j \rangle_{\mathcal{H}}^2 < +\infty \right\}
\]
and for every $\eta \in \dom\vec\Delta $,
\[
-\vec\Delta \eta=\sum_{j=1}^{+\infty}  \langle \eta , \partial \phi_j \rangle_\mathcal{H} \partial \phi_j.
\]
\end{lemma}

\begin{proof}
We observe first that
\begin{align*}
\dom\eng &=\left\{ f \in L^2(X,\mu) , \lim_{t \to 0} \left\langle \frac{f -e^{t \Delta}f}{t} , f \right\rangle_{L^2(X,\mu)} \text{  exists} \right\} \\
 &=\left\{ f \in L^2(X,\mu) , \lim_{t \to 0} \frac{1}{t} \sum_{j=1}^{+\infty} (1-e^{- \lambda_j t}) \langle f , \phi_j \rangle^2 \text{  exists} \right\} \\
  &=\left\{ f \in L^2(X,\mu), \sum_{j=1}^{+\infty} \lambda_j \langle f, \phi_j \rangle^2 <+\infty \right\}  ,
\end{align*}
and moreover that for $f \in \dom \partial=\dom\eng$, 
\begin{align}\label{bar}
\partial f = \sum_{j=1}^{+\infty}  \langle f ,  \phi_j \rangle \partial \phi_j.
\end{align}
As a consequence,
\[
\dom \partial^* =\left\{ \eta \in \mathcal{H}, \sum_{j=1}^{+\infty}  \langle \eta,\partial \phi_j \rangle_{\mathcal{H}}^2 <+\infty \right\},
\]
and for $\eta \in \dom \partial^*$,
\[
\partial^* \eta=\sum_{j=1}^{+\infty}  \langle \eta , \partial \phi_j \rangle_\irt  \phi_j.
\]
From the definition of $\vec\Delta$, this immediately yields
 \[
\dom\vec\Delta =\left\{ \eta \in \mathcal{H}, \sum_{j=1}^{+\infty} \lambda_j \langle \eta , \partial \phi_j \rangle_{\mathcal{H}}^2 < +\infty \right\}
\]
and for every $\eta \in \dom\vec\Delta $ we have,
\[
-\vec\Delta \eta=\sum_{j=1}^{+\infty}  \langle \eta , \partial \phi_j \rangle_\irt \partial \phi_j.
\]
\end{proof}

\subsection{Bakry-\'Emery estimates}

The intertwining property in Theorem \ref{intertwining1} may be used to establish  Bakry-\'Emery type estimates for $e^{t\Delta}$.

It is possible to think of $\irt$ as  measurable sections of a vector bundle over $X$. Our presentation follows \cite{HRT}, but  follows from \cite{Ebe99}.  Let $\set{f_n}_{n=1}^\infty$ be a countable set of functions which is $\eng$-dense in $\diralg$. Define $\A := \spn\set{f_n}_{n=1}^\infty$.
Define a positive bilinear form on simple tensors of $\A\otimes \B_b(X)$ by
\[
\gen{f_1\otimes g_1,f_2\otimes g_2}_{\irt_x} = g_1(x)g_2(x)\Gamma(f_1,f_2)(x).
\]
The fibre of $\irt$ at $x$ is defined to be the space $\irt_x := \A/\ker \gen{.,.}_{\irt_x}$, where 
\[
\ker\gen{.,.}_{\irt_x} := \set{\eta\in\A\otimes \B_b(X)~:~\gen{\eta,\eta}_{\irt,x} = 0}.
\]

\begin{example}
In the case of Example \ref{example1}, $\irt_x$ can be identified with $T^*_x \mathbb{M}$.

\end{example}

\begin{theorem}[\cite{HRT}, {Theorem 2.1} or \cite{Ebe99}, {Theorem 3.9}]
The fibres $\irt_x$ are a measurable field over $X$, and $\irt$ is isometrically isomorphic to $ \int_{X}^\oplus \irt_x \ d\mu(x)$. In particular, for any $\eta_1,\eta_2\in \irt$,
\[
\gen{\eta_1,\eta_2}_\mathcal{H} = \int_{X} \gen{\eta_1,\eta_2}_{\irt_x} \ d\mu(x).
\]
\end{theorem}

$\norm{.}_{\irt_x}$ shall be used to denote the fiberwise norm associated to $\gen{.,.}_{\irt_x}$. Note that, for any $a_1,a_2\in\B_b(X)$ and $\eta_1,\eta_2\in \irt$ then, $\gen{a_1\eta_1,a_2\eta_2}_{\irt,x} = a_1(x)a_2(x) \gen{\eta_1,\eta_2}_{\irt_x}$.

The following result is then easy to establish.

\begin{theorem}\label{self_gh}
Let $C_1 \ge 1$. Assume that for every $\eta \in \irt$, we have $\mu$-almost everywhere
\[
\| e^{t \vec{\Delta}} \eta \|_{\mathcal{H}_x} \le C_1 (e^{t \Delta} \| \eta \|_{\mathcal{H}_{.}})(x), \quad 0 \le t \le 1.
\]
Then, the semigroup $e^{t \Delta}$ satisfies the Bakry-\'Emery estimate
\begin{align*}
\sqrt{\Gamma(e^{t \Delta} f)} \le C_1 e^{C_2 t} e^{t \Delta} \sqrt{\Gamma(f)}, \quad f \in \dom\eng, \quad t \ge 0, 
\end{align*}
for some $C_2 \in \mathbb{R}$.
\end{theorem}

\begin{proof}
From Theorem \ref{intertwining1}, we have for $f \in \dom \eng$,
\[
\partial e^{t \Delta} f =  e^{t \vec{\Delta}} \partial f.
\]
Since $\|  \partial e^{t \Delta} f  \|_{\mathcal{H}_x}=\sqrt{\Gamma(e^{t \Delta} f) (x)}$, we deduce that for $0 \le t \le 1$, we have
\begin{align*}
\sqrt{\Gamma (e^{t \Delta} f)} \le C_1  e^{t \Delta} \sqrt{\Gamma ( f)}.
\end{align*}
Applying this inequality with $e^{t \Delta} f$ instead of $f$, we deduce that for $0 \le t \le 1$, we have
\begin{align*}
\sqrt{\Gamma (e^{2t \Delta} f)} \le C_1^2  e^{2t \Delta} \sqrt{\Gamma ( f)}.
\end{align*}
By induction, we then easily deduce that for $n \in \mathbb{N}$, $n \ge 1$, and $t \in [n-1, n]$, we have
\[
\sqrt{\Gamma (e^{t \Delta} f)} \le C_1^n  e^{t \Delta} \sqrt{\Gamma ( f)}.
\]
We conclude then that
\[
\sqrt{\Gamma (e^{t \Delta} f)} \le C_1 C_1^{n-1}  e^{t \Delta} \sqrt{\Gamma ( f)}\le C_1 e^{C_2t}e^{t \Delta} \sqrt{\Gamma ( f)},
\]
with $C_2=\ln C_1$.
\end{proof}

\begin{remark}
One can deduce from \cite{Shi97} several statements equivalent to the semigroup domination
\[
\| e^{t \vec{\Delta}} \eta \|_{\mathcal{H}_x} \le  (e^{t \Delta} \| \eta \|_{\mathcal{H}_{.}})(x), \quad 0 \le t \le 1.
\]
However, in this work, we shall be more interested in situations for which the optimal  $C_1 $ is strictly larger than 1. Such situations include the metric graphs studied in Section 5 of the present paper.
\end{remark}

\begin{example}
In the case of Example \ref{example1}, the semigroup domination
\[
\| e^{t \vec{\Delta}} \eta \|_{\mathcal{H}_x} \le  (e^{t \Delta} \| \eta \|_{\mathcal{H}_{.}})(x), \quad 0 \le t \le 1.
\]
is equivalent to the non negativity of the Ricci curvature tensor.
\end{example}

In large times, the  Bakry-\'Emery estimate may be obtained under some \textit{compactness} assumptions.

\begin{theorem}\label{BE large time}
Assume that $\Delta$ has a pure point spectrum, $ 1 \in \dom \Delta$ and that the Dirichlet space $(\eng , \dom \eng)$ satisfies the Poincar\'e inequality:
\[
\int_X\left( f-\int_X f d\mu\right)^2 d\mu \le \frac{1}{\lambda_1} \eng (f,f), \quad f \in \dom \in  \eng.
\]
Assume moreover that the heat kernel $p_t(x,y)$ of $e^{t \Delta}$ satisfies the estimates: For some $ t_0 >0$, there exists a $M>0$ such that  for $\mu$-almost every $x,y \in X$
\[
p_{t_0}(x,y) \le M, \quad |  \Gamma(p_{t_0} (.,y))(x) | \le M.
\]
Then, there exist  constants $C>0$ and $t_1>t_0$ depending only on $M$, $t_0$ and the spectrum of $\Delta$ such that for every $t >t_1$ and $f \in \dom \eng$,
\[
\sqrt{\Gamma(e^{t \Delta } f) }\le C e^{-\lambda_1 t} e^{t \Delta }\sqrt{\Gamma( f) }.
\]
\end{theorem}

\begin{proof}
From the assumptions $\mu (X) <+\infty$. For convenience, we assume that $\mu(X)=1$. From the spectral decomposition
\[
-\Delta f =\sum_{j=1}^{+\infty} \lambda_j \langle f , \phi_j \rangle \phi_j, \quad f \in \dom\Delta,
\]
where the $\lambda_i$'s are the non-zero eigenvalues, and the Poincar\'e inequality we deduce that
\[
p_t(x,y)=1+\sum_{j=1}^{+\infty} e^{-\lambda_j t} \phi_j (x) \phi_j(y).
\]
Since $e^{t_0 \Delta } \phi_j=e^{-\lambda_j t_0} \phi_j$, we deduce that $\mu$ almost everywhere
\begin{align*}
|\phi_j (x)|&=e^{\lambda_j t_0}\left|  \int_X p_{t_0} (x,y) \phi_j (y) d\mu(y) \right| \\
 &\le e^{\lambda_j t_0} \left( \int_X p_{t_0} (x,y)^2 d\mu (y) \right)^{1/2} \\
 &\le M e^{\lambda_j t_0}.
\end{align*}
Similarly, from the bound  $|   \Gamma(p_{t_0} (.,y))(x) | \le M$, one obtains
\[
\| \partial \phi_j \|_{\mathcal{H}_x} \le M e^{\lambda_j t_0}.
\]
As a first consequence, for $t >2 t_0$,
\begin{align*}
| p_t (x,y)-1|& \le \sum_{j=1}^{+\infty} e^{-\lambda_j t} | \phi_j (x) \phi_j(y)| \\
 &\le M^2  \sum_{j=1}^{+\infty} e^{-\lambda_j (t-2t_0)},
\end{align*}
and thus
\[
p_t(x,y) \ge 1-M^2  \sum_{j=1}^{+\infty} e^{-\lambda_j (t-2t_0)}.
\]
Now from \eqref{bar}, one has for $f \in \dom \eng$,
\begin{align*}
\partial e^{t \Delta } f  &  =\sum_{j=1}^{+\infty}  \langle e^{t\Delta} f , \phi_j \rangle_\mathcal{H} \partial \phi_j  \\
 &=\sum_{j=1}^{+\infty} e^{-\lambda_j t} \langle  f , \phi_j \rangle_\mathcal{H} \partial \phi_j \\
 &=- \sum_{j=1}^{+\infty} \frac{1}{\lambda_j} e^{-\lambda_j t} \langle  \partial f , \partial \phi_j \rangle_\mathcal{H} \partial \phi_j
\end{align*}
This implies that for $t>2t_0$, $\mu$ almost everywhere
\[
 \| \partial e^{t \Delta } f \|_{\mathcal{H}_x} \le M^2 \int_{X} \|   \partial  f \|_{\mathcal{H}_y} d\mu(y)    \sum_{j=1}^{+\infty}  \frac{1}{\lambda_j} e^{-\lambda_j (t-2t_0)}.
\]
We conclude that for $t$ large enough
\[
 \| \partial e^{t \Delta } f \|_{\mathcal{H}_x} \le C(t) (e^{t \Delta } \| \partial  f \|_{\mathcal{H}_\cdot} )(x)
 \]
 where
 \[
 C(t)=\frac{M^2  \sum_{j=1}^{+\infty}  \frac{1}{\lambda_j} e^{-\lambda_j (t-2t_0)}}{ 1-M^2  \sum_{j=1}^{+\infty}   e^{-\lambda_j (t-2t_0)}}.
 \]
 The conclusion easily follows.
\end{proof}

\section{Sobolev and  isoperimetric inequalities on Dirichlet spaces}

\subsection{Setting}

In this section, we  work under the assumptions of Section 2.1.  We shall moreover assume that $(X,d)$ is a compact metric space and that $\mu(X)=1$.  Let $P_t=e^{t \Delta}$ be the heat semigroup generated by $\Delta$.  We shall  assume that the semigroup on forms $\vec{P}_t=e^{t \vec{\Delta}}$ satisfies $\mu$-almost everywhere the semigroup domination:
\begin{align}\label{PLK}
 \| \vec{P}_t \eta \|_{\mathcal{H}_x} \le  C_1 (P_t  \| \eta \|_{\mathcal{H}_\cdot} )(x), \quad \eta \in \mathcal{H}, \quad  0 \le t \le 1,
\end{align}
where $C \ge 1$ is a constant. From the results in Section 5, this assumption is for instance satisfied in any metric graph which has a finite number number of edges.  It follows from Theorem \ref{self_gh} that a Bakry-\'Emery estimate
\begin{align}\label{BE-buser}
\sqrt{\Gamma(P_t f)} \le C_1 e^{C_2 t} P_t \sqrt{\Gamma(f)}, \quad f \in \dom\eng, \quad t \ge 0, 
\end{align}
is satisfied for some $C_2 \ge 0$. We note that, as in Section 6, Remark 6.6,  of \cite{BBBC}, this inequality alone automatically implies a large number of functional  inequalities for the heat semigroup. We also note that \eqref{BE-buser} is equivalent to a Lipschitz continuity  property of $P_t$ in the Wasserstein distance (see \cite{Kuw10}).  In this section, we are interested in the space of bounded variation functions and take full advantage of the semigroup domination \eqref{PLK} to prove several Sobolev and isoperimetric type inequalities.

\

In the sequel, we will say that $\eng$ satisfies a spectral gap inequality if there exists a positive $\lambda_1$, such that for every $f \in \dom \eng$,
\[
\int \left( f-\int f d\mu \right)^2 d\mu \le \frac{1}{\lambda_1} \eng (f,f)
\]
The best constant $\lambda_1$ in this inequality is then called the spectral gap.

\

We will say that $\eng$ satisfies a log-Sobolev inequality if there exists a positive $\rho_0$, such that for every $f \in \dom \eng$,
\begin{equation}\label{log Sob}
\int f^2 \ln f^2 d\mu -\int f^2 d\mu \ln  \int f^2 d\mu \le \frac{1}{\rho_0}  \int \Gamma(f) d\mu,
\end{equation} 

The best constant $\rho_0$ in this inequality is then called the log-Sobolev constant.

\

Criteria to ensure that a spectral gap and a log-Sobolev inequality are well known.
Define $p_t(x,y) : \R^+ \times X \times X \to \R$ be the heat kernel of $\Delta$ if it is the integral kernel of $P_t$, that is $P_tf (x) = \int_X f(y)p_t(x,y) \ dy$. We will assume that such kernel exists, is jointly continuous in $(t,x,y)$ and that for some $t>0$, $ \inf_{x,y} p_t (x,y) >0$. 

A first consequence of those assumptions is the regularization property of $P_t$. If $f \in L^2(X)$, then the integral $\int p_t(x,y)f(y)d\mu(y)$, $t>0$, is convergent for every $x \in X$ since
\begin{align*}
\int p_t(x,y)|f(y)|d\mu(y)  
 &\le  \sqrt{ \int p_t(x,y)^2 d\mu(y) } \|f \|_2 \\
 & \le \sqrt{p_{2t}(x,x)}  \|f \|_2.
\end{align*}
As a consequence, one can define $P_tf(x)$ for every $x \in X$. Observe also that $P_tf$ is then actually a bounded continuous function since
\begin{align*}
| P_t f (x) -P_tf(y) | & \le \int |p_t(x,z)-p_t(y,z)| |f(z)| d\mu(z) \\
 &   \le \sqrt{\int |p_t(x,z)-p_t(y,z)|^2d\mu(z)}  \|f \|_2 \\
 & \le \sqrt {p_{2t} (x,x) -p_{2t}(x,y) +p_{2t}(y,y)}\|f \|_2
\end{align*}
and similarly
\begin{align*}
|P_tf(x)| \le \sqrt{p_{2t}(x,x)}  \|f \|_2 \le \sup_{x \in X}\sqrt{p_{2t}(x,x)}  \|f \|_2 .
\end{align*}

Under our assumptions, both the spectral gap and the log-Sobolev inequality are actually satisfied. Indeed, the previous computation also shows  that $P_t$ is supercontractive, i.e for every $t>0$,  $\| P_t \|_{2 \to 4} <\infty$. Therefore from Gross' theorem (e.g. \cite[theorem 5.2.3]{BGL} and \cite[theorem 2.2.3]{Da} ), a defective logarithmic Sobolev inequality is satisfied, that is there exist two constants $A,B>0$  such that
\[
\int f^2 \ln f^2 d\mu -\int f^2 d\mu \ln  \int f^2 d\mu \le A \int \Gamma(f) d\mu+B\int f^2 d\mu.
\]
Since the heat kernel is positive and the invariant measure a probability, we deduce from the uniform positivity improving property (see \cite{MR1641566}, Theorem 2.11) that $\Delta$ admits a spectral gap.  That is, a Poincar\'e inequality is satisfied. It is then classical (see \cite{MR1845806}), that the conjunction of a spectral gap and a defective logarithmic Sobolev inequality implies the log-Sobolev inequality (i.e. we may actually take $B=0$ in the above inequality).  

As a consequence, for instance, the compact metric graphs considered in Section 5 satisfy a spectral gap and a log-Sobolev inequality.

\subsection{Bounded variation functions and Sobolev embeddings}

In this section, as a preliminary, we prove some results about the theory of bounded variation functions and Cacciopoli sets in our framework. Define
\[
\var f = \sup\set{ \gen{f,\partial^*  \eta }_2~|~  \eta  \in \dom \partial^*,  \| \eta \|_{\mathcal{H}_x}  \leq 1,  \mu \text{-almost everywhere}}
\]
Define the set of functions of finite variation 
\[
BV(X) := \set{f\in \B_b(X)~|~\var(f) < \infty}.
\]
This allows one to define the perimeter of a Borel set $E\subset X$ with $\mathbf 1_E\in BV(X)$ by
\[
P(E)  = \var \mathbf 1_E.
\]
Call any set $E$ with $P(E) < \infty$ a Caccioppoli set.

\begin{lemma}
If $f\in \dom\eng$, one has
\[
\frac{1}{C_1} \int \sqrt{\Gamma (f)} \ d\mu \le \var f \le  \int \sqrt{\Gamma (f)} \ d\mu.
\]
\end{lemma}

\begin{proof}
\begin{align*}
\var f & = \sup\set{ \gen{f,\partial^*  \eta }_2~|~  \eta  \in \dom \partial^*,  \| \eta \|_{\mathcal{H}_x}  \leq 1,  \mu \text{-almost everywhere}} \\
&  \le  \int \norm{\partial f}_{\irt_x}  \ d\mu(x) = \int \sqrt{\Gamma (f)} \ d\mu.
\end{align*}
On the other side, let
\[
\eta=\frac{1}{C_1} e^{s \vec{\Delta}} \left( \frac{\partial f}{ \norm{\partial f}_{\irt_\cdot} +\varepsilon} \right)
\]
where $s,\varepsilon >0$. We note that $\eta \in \dom \partial^*$. Indeed  $ \frac{\partial f}{ \norm{\partial f}_{\irt_\cdot} +\varepsilon} \in \mathcal{H}$ and thus by spectral theory $\eta \in \dom\vec\Delta\subset \dom \partial^*$ . Also,  from our assumption \eqref{PLK}, $\| \eta \|_{\mathcal{H}_x}  \leq 1$. As a consequence,
\[
\var f \ge \gen{f,\partial^*  \eta }_2=  \gen{\partial f,  \eta }_\irt= \frac{1}{C_1} \gen{\partial f,  e^{s \vec{\Delta}} \left( \frac{\partial f}{ \norm{\partial f}_{\irt_\cdot} +\varepsilon} \right) }_\irt .
\]
Letting $ s \to 0$ and then $\varepsilon \to 0$ finishes the proof. 
\end{proof}

\begin{theorem}\label{BV Hino}
Let $f \in BV(X)$.
\[
\var f \le  \lim \inf_{t \to 0} \int \sqrt{\Gamma (e^{t\Delta} f)} \ d\mu \le  \lim \sup_{t \to 0} \int \sqrt{\Gamma (e^{t\Delta} f)} \ d\mu \le C^2_1 \var f
\]
\end{theorem}

\begin{proof}
Let $f \in BV(X)$ and $\eta  \in \dom \partial^*,  \| \eta \|_{\mathcal{H}_x}  \leq 1 \text{ almost everywhere}$. We have 
\[
\gen{ e^{t\Delta} f, \partial^* \eta }_2=\gen{\partial e^{t\Delta} f, \eta }_\irt \le  \int \sqrt{\Gamma (e^{t\Delta} f)} \ d\mu.
\]
Since we have in $L^2$, $\lim_{t \to 0} e^{t\Delta} f =f$ , we deduce that
\[
\gen{  f, \partial^* \eta }_2 \le \lim \inf_{t \to 0} \int \sqrt{\Gamma (e^{t\Delta} f)} d\mu, 
\]
and therefore
\[
\var f \le  \lim \inf_{t \to 0} \int \sqrt{\Gamma (e^{t\Delta} f)}d\mu.
\]
By symmetry of $e^{t\Delta}$,
\[
\gen{ e^{t\Delta} f, \partial^* \eta }_2=\gen{ f,e^{t\Delta}  \partial^* \eta }_2=\gen{ f, \partial^* e^{t\vec{\Delta}}     \eta}_2.
\]
Where the last equality is because $\partial^*e^{t\vec\Delta}\eta = e^{t\Delta}\partial^*\eta$ from the same logic as the proof of theorem \ref{intertwining1}. (i.e. \cite[Theorem 3.1]{Shi00}).

We now observe that $e^{t\vec{\Delta}}     \eta \in \dom\vec\Delta\subset  \dom \partial^*,$ and $ \| e^{t\vec{\Delta}}     \eta \|_{\mathcal{H}_x} \leq C_1 \text{ almost everywhere}$. Therefore
\[
\gen{ e^{t\Delta} f, \partial^* \eta }_2 \le C_1 \var f,
\]
which implies
\[
\int \sqrt{\Gamma (e^{t\Delta} f)}d\mu  \le C_1 \var (e^{t\Delta} f) \le C^2_1  \var f,
\]
and of course
\[
\lim \sup_{t \to 0} \int \sqrt{\Gamma (e^{t\Delta} f)}d\mu \le C^2_1 \var f.
\]
\end{proof}

The following inequality is the cornerstone of the section.

\begin{theorem}\label{BVX}
Assume that \eqref{BE-buser} is satisfied with $C_2=0$. Let $f \in BV(X)$. For $ t\ge 0$,
\[
\| P_t f -f \|_1 \le C_1^3 \sqrt{2t} \var f.
\]
\end{theorem}

\begin{proof}
We adapt an argument due to Ledoux (see \cite[p. 953]{Le}). Let $f \in  \B_b(X)\cap \dom\eng$. From the Bakry-\'Emery estimate, we have
\[
P_t(f^2)-(P_tf)^2=2\int_0^t P_s ( \Gamma(P_{t-s} f)ds \ge \frac{2}{C_1^2} t \Gamma (P_t f).
\]
The first equality follows from the fact that
\[
\frac{d}{ds}\paren{ P_s(P_{t-s}f)^2} = P_s\Gamma(P_{t-s}f).
\] 
To prove this, let $\phi$ be a positive continuous uniformly bounded function. Then, for any given $0<s\leq t$ $P_s\phi$ and $P_{t-s}f$ are bounded from the previous section. Further, $\Gamma(P_{t-s} f)$ exists for $0\leq s \leq t$ because $P_{t-s}f \in \mathcal B_b(X)\cap \dom\eng$. Thus we have    
\begin{align*}
\frac{d}{ds} \int \phi P_s(P_{t-s}f)^2 \ d\mu & = \frac{d}{ds} \int P_s\phi (P_{t-s}f)^2 \ d\mu \\ 
& = \int (\Delta P_s \phi ) (P_{t-s}f)^2 - 2(P_s\phi)(\Delta P_{t-s}f)(P_{t-s} f) \ d\mu \\
& =  -\eng(P_s\phi, (P_{t-s}f)^2)+2\eng(P_s\phi P_{t-s}f,P_{t-s}f) \\
& = 2\int (P_s \phi) \Gamma(P_{t-s}f) \ d\mu.
\end{align*}
Since this is true for all such $\phi$, this implies that $P_s\Gamma(P_{t-s}f)$ exists and is equal to $\frac{d}{ds}P_s(P_{t-s}f)^2$ where the derivative is taken in $L^2$.

We have thus deduced that
\[
\Gamma (P_t f)\le \frac{C^2_1}{2 t} ( P_t(f^2)-(P_tf)^2).
\]
This implies
\[
\| \sqrt{\Gamma (P_tf)}\|_\infty \le  \frac{C_1}{ \sqrt{2t}} \| f \|_\infty.
\]
Let now $g \in \dom\eng$, with $g\ge 0$
and $||g||_{\infty}\le 1$, and  $f \in \B_b(X)\cap \dom\eng$.  We have
\begin{align*}
\int  g(f-P_tf) d\mu & = \int_0^t \int  g \frac{\partial P_sf}{\partial s}
d\mu ds = \int_0^t \int  g \Delta P_sf d\mu ds  = - \int_0^t \int \Gamma(P_sg,f) d\mu ds
\\
& \le \int_0^t  \| \sqrt{\Gamma(P_sg)} \|_{\infty}\int
\sqrt{\Gamma(f)} d\mu ds  \le  C_1  \sqrt{ 2 t}
\int \sqrt{\Gamma(f)} d\mu.
\end{align*}
By the regularity of $\eng$ and pointwise approximation,  this is true for every bounded Borel $g$ with $\norm{g}_\infty \leq 1$, and thus
\[
\| P_t f -f \|_1 \le C_1  \sqrt{2t}\int \sqrt{\Gamma(f)} d\mu.
\]
Let now $f \in BV(X)$. For $s>0$, we have $P_sf \in \B_b(X)\cap \dom\eng$. Thus we deduce
\[
\| P_{s+t} f -P_sf \|_1 \le  C_1 \sqrt{2t}\int \sqrt{\Gamma(P_sf)} d\mu.
\]
Taking the limit when $s \to 0$ finishes the proof.
\end{proof}

\begin{remark}
If we do not assume $C_2=0$, then the inequality
\[
\| P_t f -f \|_1 \le C_1^3 \sqrt{2t} \var f.
\]
only holds for $0 \le t \le 1$.
\end{remark}

The previous theorem has many implications in terms of Sobolev embedding theorems. Ledoux proves in Section 2 of \cite{Le2} that a $L^1$-bound of the type
\[
\| P_t f -f \|_1 \le C \sqrt{t} \var f,
\]
implies, in very general frameworks, improved Sobolev embeddings involving Besov norms. In particular we obtain  the following result:

 \begin{corollary}\label{Sobo embed}
Assume that \eqref{BE-buser} is satisfied with $C_2=0$ and that $e^{t \Delta}$ has a heat kernel $p_t(x,y)$ that satisfies for some constant $Q >1$,
\begin{align}\label{ultra}
\sup_{x,y \in X, t \in (0,1]} t^{Q/2} p_t(x,y) <+\infty.
\end{align}
Then, there exists a constant $C_Q >0$, such that for every $f \in BV(X)$,
\begin{align}\label{Sobolev}
\| f \|_p \le C_Q (\var f +\|f\|_1)  ,
\end{align}
where $p=\frac{Q}{Q-1}$.
\end{corollary}

\begin{remark}
It is consequence of the celebrated Varopoulos'  theorem that the heat kernel bound \eqref{ultra} alone implies the Sobolev inequality
\begin{align*}
\| f \|_p \le C_Q (\eng (f,f) +\|f\|_2)  ,
\end{align*}
where $p=\frac{2Q}{Q-1}$. The assumptions $C_1<+\infty, C_2=0$ are therefore used to improve this inequality into \eqref{Sobolev}.
\end{remark}

\subsection{Isoperimetric inequality}

The inequality \eqref{Sobolev} obviously has an isoperimetric flavor when applied to $f=1_E$, where $E$ is a Caccioppoli set.  Actually, adapting some beautiful ideas of Varopoulos (see \cite{Varo}, pp.256-258), Ledoux
(see pp. 22 in \cite{ledoux-bourbaki}, see also Theorem 8.4 in \cite{ledoux-stflour}) and \cite{BBnon} yields:

\begin{theorem}[Isoperimetric inequality]\label{T:iso} 
Assume that $e^{t \Delta}$ has a heat kernel $p_t(x,y)$ that satisfies for some constant $Q >1$,
\begin{align}\label{ultra2}
\sup_{x,y \in X, t \in (0,1]} t^{Q/2} p_t(x,y) <+\infty.
\end{align}
There exist constants $C_{\emph{iso}}, \mu_{\emph{max}} >0$, such that for every
Caccioppoli set $E\subset X$ with $\mu(E) \le  \mu_{\emph{max}}$
\[
\mu(E)^{\frac{Q-1}{Q}} \le C_{\emph{iso}} P (E).
\]
\end{theorem}

\begin{proof}
Let $f \in BV(X)$. From the proof of Theorem \ref{BVX}, one has for $ 0 \le t\le 1$,
\[
\| P_t f -f \|_1 \le C_1^3 \sqrt{2t} \var f.
\]
Therefore, if $E$ is a Caccioppoli set
\[ \|P_t \mathbf 1_E -
\mathbf 1_E\|_{1} \le   C_1^3 \sqrt{2t} 
\var  (\mathbf 1_E) =  C_1^3 \sqrt{2t} \ P(E),\ \
\]
Observe now that, because $P_E\mathbf 1_E \leq 1$ on $E$ and  we have
\begin{align*}
||P_t \mathbf 1_E - \mathbf 1_E||_{1}   =& \int_E (1-P_t 1_E) d\mu + \int_{E^c} P_t(1_E) d\mu\\
                              =& \int_E (1-P_t 1_E) d\mu + \int_E(P_t 1_{E^c}) d\mu\\
                              =& 2 \left( \mu(E)- \int_E P_t (1_E) d\mu \right)
\end{align*}
On the other hand, we have
\[
\int_E  P_t \mathbf 1_E d\mu  = \int \left(P_{t/2}\mathbf
1_E\right)^2 d\mu.
\]
We thus obtain
\begin{align}\label{eqn:escapettime}
||P_t \mathbf 1_E - \mathbf 1_E||_{1} = 2 \left(\mu(E) -
\int  \left(P_{t/2}\mathbf 1_E\right)^2 d\mu\right).
\end{align}
We now note that 
\begin{align*}
\int (P_{t/2} \mathbf 1_E)^2 d\mu & \le \left(\int_E
\left(\int p_{t/2}(x,y)^2
d\mu(y)\right)^{\frac{1}{2}}d\mu(x)\right)^2
\\
& = \left(\int_E p_t(x,x)^{\frac{1}{2}}d\mu(x)\right)^2 \le
\frac{A}{t^{Q/2}} \mu(E)^2.
\end{align*}
for some constant $A>0$.
Combining these equations yields
\[
\mu(E)  \le   B  \sqrt{ t} \ P(E) +
\frac{C}{t^{Q/2}} \mu(E)^2,\ \ \ \  0 <t \le 1,
\]
for some positive constants $B,C$. Applying the inequality at $t=D \mu(E)^{2/Q}$ where $D$ is large enough concludes the proof.
\end{proof}

\begin{remark}
We note that we do not assume $C_2=0$ in the theorem.
\end{remark}
\subsection{Buser's isoperimetric inequality}

In the context of  a smooth compact  Riemannian manifold with Riemannian measure $\mu$, Cheeger  \cite{Ch}  introduced   the following isoperimetric constant
\[
h=\inf \frac{\mu (\partial A)}{\mu (A)},
\]
where the infimum runs over all open subsets $A$ with smooth boundary $\partial A$ such that $\mu(A)\le \frac{1}{2}$. Cheeger's constant can be used to bound from below the first non zero eigenvalue of the manifold. Indeed, it is proved in \cite{Ch} that
\[
\lambda_1 \ge \frac{h^2}{4}.
\]

Buser \cite{Bu} then proved that if the Riemannian Ricci curvature of the manifold is non-negative, then we actually have
\[
\lambda_1 \le C h^2
\]
where $C$ is a universal constant depending only on the dimension. Buser's inequality was reproved by Ledoux \cite{Le} using heat semigroup techniques. Under proper assumptions, by using the tools we introduced, Ledoux' technique can essentially reproduced in our general framework of Dirichlet spaces.

In this section, we assume that $\eng$ satisfies a spectral gap inequality and that \eqref{BE-buser} is satisfied with $C_2=0$. We define the Cheeger's constant of $X$ by
\[
h=\inf \frac{P(E)}{\mu(E)}
\]
where the infimum runs over all Caccioppoli sets $E$ such that $\mu(E)\le \frac{1}{2}$. We denote by $\lambda_1$ the spectral gap of $\Delta$.

\begin{theorem}\label{buser}
\[
\lambda_1 \le C_{\text{buser}} h^2,
\]
where $C_{\text{buser}}$ is a constant depending on $C_1$ only.
\end{theorem}

\begin{proof}
Let $A$ be a Caccioppoli  set with finite perimeter. By symmetry and stochastic completeness of the semigroup, we have from equation (\ref{eqn:escapettime})
\begin{align*}
\Vert 1_A - P_t 1_A \Vert_1 
                              =& 2  \left( \mu(A)- \Vert P_\frac{t}{2} (1_A) \Vert_2^2 \right).
\end{align*}
By Theorem \ref{BVX}, we have
\[
\| P_t 1_A -1_A \|_1 \le C_1^3 \sqrt{2t} P(A).
\]
We deduce that
\[
\mu(A) \le C_1^3 \sqrt{\frac{t}{2}} P(A)+ \Vert P_\frac{t}{2} (1_A) \Vert_2^2.
\]
Now, by spectral theorem,
\[
\Vert P_\frac{t}{2} (1_A) \Vert_2^2=\mu(A)^2 + \Vert P_\frac{t}{2} (1_A-\mu(A)) \Vert_2^2 \le \mu(A)^2 +e^{-\lambda_1 t} \| 1_A-\mu(A) \|_2^2
\]
This yields
\[
\mu(A) \le C_1^3 \sqrt{\frac{t}{2}} P(A)+ \mu(A)^2 +e^{-\lambda_1 t} \| 1_A-\mu(A) \|_2^2.
\]
Equivalently, one obtains
\[
C_1^3 \sqrt{\frac{t}{2}} P(A) \ge \mu(A) (1 -\mu(A))(1-e^{-\lambda_1 t}).
\]
Therefore,
\[
h \ge \frac{1}{C_1^3 \sqrt{2}} \sup_{t >0} \left( \frac{1-e^{-\lambda_1 t}}{\sqrt{t}} \right).
\]
We conclude
\[
h^2 \ge \frac{1}{2C_1^6 } (1-e^{-1})^2  \lambda_1.
\]
\end{proof}

Let us observe that it is known that the Cheeger lower bound on $\lambda_1$ may be obtained under further assumptions on the Dirichlet space $(X,d,\eng)$. Indeed, assume that Lipschitz functions are in the domain of $\eng$ and that $\sqrt{\Gamma(f)}$ is an upper gradient in the sense that for any Lipchitz function $f$,
\[
\sqrt{\Gamma(f)}(x) =\lim \sup_{d(x,y)\to 0} \frac{ |f(x)-f(y)|}{d(x,y)}.
\]
In that case, if $A$ is a closed set of $X$, one defines its Minkowski exterior boundary measure by
\[
\mu_+(A)=\lim \inf_{\varepsilon \to 0} \frac{1}{\varepsilon} \left( \mu(A_\varepsilon) -\mu(A) \right),
\]
where $A_\varepsilon=\{ x \in X, d(x,X) <\varepsilon \}$. We can then define the second Cheeger's constant of $X$ by
\[
h_+=\inf \frac{\mu_+(E)}{\mu(E)}
\]
where the infimum runs over all closed sets $E$ such that $\mu(E)\le \frac{1}{2}$. Then, according to Theorem 8.5.2 in \cite{BGL}, one has
\[
\lambda_1 \ge \frac{h^2_+}{4}.
\]

\subsection{Ledoux isoperimetric inequality}

In this section, we assume that $\eng$ satisfies a log-Sobolev inequality and that \eqref{BE-buser} is satisfied with $C_2=0$. We define the Gaussian isoperimetric constant of $X$ by
\[
k=\inf \frac{P(E)}{\mu(E)\sqrt{-\ln \mu(E)}}
\]
where the infimum runs over all Caccioppoli sets $E$ such that $\mu(E)\le \frac{1}{2}$. We denote by $\rho_0$ the log-Sobolev constant of $X$, that is the best constant in the inequality \eqref{log Sob}.

\begin{theorem} 

\[
\rho_0 \le C_{\text{ledoux}} k^2
\]
where $C_{\text{ledoux}}$ is a constant depending on $C_1$ only.
\end{theorem}

\begin{proof}
Let $A$ be a Caccioppoli  with finite perimeter.  From the proof of Theorem \ref{buser}, we have
\[
\mu(A) \le C_1^3 \sqrt{\frac{t}{2}} P(A)+ \Vert P_\frac{t}{2} (1_A) \Vert_2^2.
\]
Now we can use the hypercontractivity constant to bound $\Vert P_\frac{t}{2} (1_A) \Vert_2^2$. Indeed, from Gross' theorem it is well known  that the logarithmic Sobolev inequality  
\[
\int f^2 \ln f^2 d\mu -\int f^2 d\mu \ln  \int f^2 d\mu \le \frac{1}{\rho_0}  \int \Gamma(f) d\mu,
\]
is equivalent to hypercontractivity property
$$
\Vert P_t f \Vert_q \leq \Vert f \Vert_p 
$$
 for all $f$ in $L^p$ whenever $1<p<q<\infty$ and $e^{\rho_0 t}\geq \sqrt \frac{q-1}{p-1}$.

Therefore, with $p(t)= 1+e^{-2\rho_0 t}<2$, we get,
\begin{align*}
C_1^3  \sqrt{2t} P(A) \geq & 2  \left( \mu(A)- \mu(A)^\frac{2}{p(t)} \right)\\
          &\geq  2 \mu(A) \left(1- \mu(A)^\frac{1-e^{2-\rho_0 t}}{1+e^{-2\rho_0 t}} \right).
\end{align*}
By using then the computation page 956 in \cite{Le}, one deduces that if $A$ is a set which has a finite perimeter $P(A)$ and such that $0\leq \mu(A)\leq \frac{1}{2}$, then 
$$
P( A)\geq \tilde{C} \sqrt{ \rho_0}   \mu(A)\left(\ln \frac{1}{\mu(A)} \right)^\frac{1}{2},
$$
where $\tilde{C}$ is a constant depending on $C_1$ only.
\end{proof}

\section{Poincar\'e Duality  on Hino index-1 spaces}

In this section, we come back to the general framework of Section 2.1. Our goal is to construct a scalarization of the closed symmetric form
\[
\vec{\eng} (\omega, \eta)=\langle \partial^* \omega, \partial^* \eta \rangle_2.
\]
This can be achieved on Hino index-1 spaces where one-forms may be identified with functions. In such spaces, we will prove that the semigroup domination
 \begin{align*}
 \| \vec{P}_t \eta \|_{\mathcal{H}_x} \le  C_1 (P_t  \| \eta \|_{\mathcal{H}_\cdot} )(x), \quad \eta \in \mathcal{H}, \quad  0 \le t \le 1,
\end{align*}
is then equivalent to a semigroup domination
\[
| e^{t \Delta^\perp} f | (x)\le C_1 e^{t \Delta} |f | (x), \quad 0 \le t \le 1.
\]
where $\Delta^\perp$ is a self adjoint operator on $L^2(X)$ that we call Poincar\'e dual of $\Delta$. We stress that $\Delta^\perp$, in general, is not Markovian, that is the semigroup $e^{t \Delta^\perp}$ is not positivity preserving.

\subsection{Poincar\'e duality}



We first recall the following definition.

\begin{defn}[Definition 2.9 in \cite{Hin10}]
The pointwise Hino index $p(x)$ of $(\eng,\dom\eng)$ is the function such that
\begin{enumerate}[(a)]
\item For any $N \in \N$ and $f_1,\ldots, f_N \in \dom\eng$ the rank of the $N\times N$ matrix with entries $Z_{ij}:=\Gamma (f_i,f_j)$ is less than $p(x)$ almost everywhere.
\item If $p'(x)$ is another function which satisfies (a), then $p(x) \leq p'(x)$ almost everywhere. 
\end{enumerate}
The essential supremum of $p(x)$ with respect to $\mu$ is referred to as the Hino index of $(\eng,\dom\eng)$. 
\end{defn}

\begin{prop}[Lemma 3.2 in \cite{Hin13}]\label{prop:HinoIndex}
If $p(x)$ is the pointwise Hino index of $(\eng,\dom\eng)$ then $p(x) = \dim\irt_x$ almost everwhere.
\end{prop}


For $\omega \in \irt$, define $\nu_\omega$ to be the measure on $X$ such that for $\phi\in C_b(X)$
\[
\int_X \phi \ d\nu_\omega = \gen{\omega \cdot \phi,\omega}_\mathcal{H}.
\]
The following lemma is then trivial.

\begin{lemma}
There exists $\omega \in \mathcal{H}$ such that $\mu=\nu_\omega$ if and only if there exists $\omega \in \mathcal{H}$ such that $\| \omega \|_{\mathcal{H}_x}=1$, $\mu$-a.e.
\end{lemma}

We now set the following definition of the Hodge star operator on Hino index-1 spaces.

\begin{defn}
Assume that $(\eng,\dom\eng)$ has Hino index 1 and that $\mu=\nu_\omega $ for some  $\omega \in \irt$. For $\mu$-almost every $x\in X$, we define the Hodge star operator $\star: L^2(X,\mu) \to \irt = \int^\oplus \irt_x \ d\mu$ by $\star f$ which is defined to be $(\star f)_x : = f(x)\omega_x$ on almost every fiber $\irt_x$ of $\irt$. $\star$ shall also be used to denote the inverse of this map $\star (\omega\cdot f) = f$.
\end{defn} 

Classically, for $n$-dimensional Riemannian manifolds, Poincar\'e duality states the differential $p$ forms are isometric to $n-p$ forms, and this isometry is given by the Hodge star. For 1-dimensional spaces (i.e. the line or the circle), the classical Hodge star provides an isometry between $0$ forms (functions) and $1$ forms. Hence, the following proposition is a measurable version of Poincar\'e duality for 1 dimensional spaces.

\begin{prop}
Assume that $(\eng,\dom\eng)$ has Hino index 1 and that $\mu=\nu_\omega $ for some  $\omega \in \irt$. The operator $\star$ is an isometry both fibre-wise and globally. i.e. $\norm{\star f}_{\irt_x} = |f(x)|$ almost everywhere, and $\norm{\star f}_{\irt} = \norm{f}_2$. Thus $\irt$ is isometric to $L^2(X,\mu)$.
\end{prop}

\begin{proof}
This holds because
\[
\gen{\star f,\star g}_{\irt,x} = \gen{f(x) \cdot \omega_x,g(x)\cdot\omega_x}_{\irt_x} = f(x)g(x)
\]
almost everywhere and
\[
\gen{\star f,\star g}_{\irt} = \int \gen{f\omega ,g\omega}_{\irt_x}^2 \ d\mu(x) = \int f(x)g(x) \ d\mu(x).
\]
\end{proof}

\begin{definition}
Assume that $(\eng,\dom\eng)$ has Hino index 1 and that $\mu=\nu_\omega $ for some  $\omega \in \irt$. The self-adjoint operator
\[
\Delta^\perp=\star \vec{\Delta} \star
\]
will be called the (Poincar\'e) dual operator of $\Delta$. It is the self-adjoint  generator of the closed symmetric form on $L^2(X,\mu)$
\[
\eng^\perp(f,g)=\langle \partial^* \star f , \partial^* \star g\rangle_2.
\]
\end{definition}

\begin{examples}
\begin{enumerate}
\item Let $X=\mathbb{R}$ or $X=\mathbb{S}^1$. Consider the standard Dirichlet form on $X$ which is the closure of 
\[
\mathcal{E}(f,g)=\int_X f'(x)g'(x) dx, \quad f,g \in C_0^\infty(X).
\]
Then, $(\eng^\perp, \dom \eng^\perp)=(\eng,\dom\eng)$ and $\Delta^\perp=\Delta$.
\item Let $X=I$, where $I$ is an interval of $\mathbb{R}$. Denote $(\eng_D,\dom\eng_D)$ the standard Dirichlet form  $\int_X f'(x)g'(x) dx$ with Dirichlet boundary condition, and denote  $(\eng_N,\dom\eng_N)$ the one with Neumann boundary condition. Then,
\[
 (\eng_N^\perp, \dom \eng_N^\perp)=(\eng_D,\dom\eng_D).
\]
Therefore,
\[
 \Delta_N^\perp=\Delta_D.
\]
This duality between the Dirichlet and the Neumann boundary conditions is exceptional --- In general $(\eng^\perp,\dom\eng^\perp)$ is not a Dirichlet form, since it may fail to satisfy the Markovian property, as is the case with the metric graphs in Section 5.1, 5.2 (see in particular the Walsh spider, Example 5.1). 
\end{enumerate}
\end{examples}

We are interested in $(\eng^\perp,\dom\eng^\perp)$ because of the following intertwining property:

\begin{theorem}\label{intertwining p}
Assume that $(\eng,\dom\eng)$ has Hino index 1 and that $\mu=\nu_\omega $ for some  $\omega \in \irt$. For $f \in \dom \eng$,
\[
\star  \partial e^{t\Delta} f=e^{t\Delta^\perp} \star \partial f, \quad t \ge 0.
\]
\end{theorem}

\begin{proof}
From Theorem \ref{intertwining1}, one has
\[
\partial e^{t \Delta}=e^{t \vec{\Delta}}\partial.
\]
Thus,
\[
\star \partial e^{t \Delta}=\star e^{t \vec{\Delta}}\partial.
\]
Since $\star$ is an isometry one has
\[
\star e^{t \vec{\Delta}} \star=e^{t\Delta^\perp},
\]
and the conclusion easily follows.
\end{proof}

The following corollary is then obvious:

\begin{corollary}\label{CR1}
Let $C_1 \ge 1$. Assume that for every $f \in L^2$, we have $\mu$-almost everywhere
\[
| e^{t \Delta^\perp} f | \le C_1 e^{t \Delta} |f |, \quad 0 \le t \le 1.
\]
Then, the semigroup $e^{t \Delta}$ satisfies the Bakry-\'Emery estimate
\begin{align*}
\sqrt{\Gamma(e^{t \Delta} f)} \le C_1 e^{C_2 t} e^{t \Delta} \sqrt{\Gamma(f)}, \quad f \in \dom\eng, \quad t \ge 0, 
\end{align*}
for some $C_2 \ge 0$.
\end{corollary}

\subsection{Harmonic forms }

A form $\omega$ in $\irt$ is called  harmonic if $\partial^* \omega=0$.  In this subsection we assume that $\eng$ has Hino index 1 and we consider the Hodge star $\star$ with respect to a harmonic form.


\begin{lemma}
Assume that $(\eng,\dom\eng)$ has Hino index 1 and that $\mu=\nu_\omega $ for some harmonic form $\omega \in \irt$. Then, for every $f,g \in \dom \eng$,
\[
\langle f, \star \partial g \rangle_2=-\langle \star \partial f, g\rangle_2.
\]
Therefore, $(\eng^\perp,\dom\eng^\perp)$ is an extension of $(\eng,\dom\eng)$.
\end{lemma}

\begin{proof}
If $f,g \in \mathcal{C}$,
\begin{align*}
\langle f , \star \partial g \rangle_2=\langle \star f ,  \partial g \rangle_\mathcal{H}=\langle  f \omega ,  \partial g \rangle_\mathcal{H}=\langle   \omega , f \partial g \rangle_\mathcal{H}=-\langle   \omega , g \partial f \rangle_\mathcal{H}=-\langle \star \partial f, g\rangle_2,
\end{align*}
because $\partial (fg)=f \partial g +g\partial f$ and $\partial^*\omega=0$. The identity extends to every $f,g \in \dom \eng$ by regularity of $\eng$ as follows: for $f\in\dom\eng$ we can find a sequence of $f_i$ in $\mathcal C$ with $\lim\eng(f_i-f) + \norm{f_i - f}_2 = 0$. For $g\in\dom\eng$, $\lim_{i\to\infty}\gen{f_i,\star \partial g}_2 = \gen{f,\star \partial g}$ because $f_i$ converges to $f$ in $L^2(X)$. On the other hand $\norm{\partial(f_i-f)}_\irt = \sqrt{\eng(f_i-f)} \to 0$. This implies,  $\lim_{i\to\infty } \partial f_i =\partial f$ strongly in $\irt$, and thus $\lim_{i\to \infty} \gen{\partial f_i, \star g}_\irt = \gen{\partial f, \star g}_\irt$.
\end{proof}

From the previous proposition we have $\star \partial \subset -\partial^* \star$. However, in general  it is not true that $\star \partial = -\partial^* \star$ (see the following discussion on the union of circles for an example).  In the case, where $\star \partial = -\partial^* \star$, then $\eng=\eng^\perp$ and therefore $\Delta=\Delta^\perp$, which implies from Corollary \ref{CR1} that the Bakry-\'Emery estimate is satisfied with a constant 1. In general, one can prove the Bakry-\'Emery estimate with constant 1 only on a subspace of $\dom \eng$.

\begin{theorem}\label{lkmn}
Let 
\[
\mathfrak{L}=\left\{ f \in L^2(X,\mu), \text{ for every } t \ge 0, e^{t\Delta } f =e^{t\Delta^\perp }f  \right\}.
\]
Then  $\mathfrak{L}$ is an $L^2$-closed subspace $\mathfrak{L}$ of $L^2(X,\mu)$  such that $ \star \partial ( \mathfrak{L} \cap \dom \eng ) \subset \mathfrak L$ and for every $f \in \mathfrak{L} \cap \dom \eng $ and $t \ge 0$,
\[
\sqrt{\Gamma(e^{t\Delta} f )} \le e^{t\Delta}\sqrt{ \Gamma(f)}.
\]
\end{theorem}

\begin{proof}
The fact that $\mathfrak{L}$ is an $L^2$-closed subspace $\mathfrak{L}$ of $L^2(X,\mu)$  is obvious. Let now $f \in \mathfrak{L} \cap \dom \eng$. We have $ e^{t\Delta } f =e^{t\Delta^\perp }f $. Therefore $\star \partial e^{t\Delta } f =\star \partial e^{t\Delta^\perp }f $. Now, from Theorem \ref{intertwining p}, $\star \partial e^{t\Delta } f = e^{t\Delta^\perp } \star \partial f$. On the other hand, from the previous lemma $\star \partial e^{t\Delta^\perp }f =-\partial^* \star e^{t\Delta^\perp }f =-\partial^* e^{t \vec{\Delta}} \star f=- e^{t \Delta} \partial^* \star f=e^{t \Delta} \star \partial f$. We conclude $e^{t\Delta^\perp } \star \partial f=e^{t\Delta } \star \partial f$ and thus $\star \partial f \in \mathfrak L$. Finally, if $f \in \mathfrak L \cap \dom \eng$, then
\[
\star \partial e^{t\Delta } f= e^{t\Delta } \star \partial f,
\]
which immediately implies the Bakry-\'Emery estimate.
\end{proof}

We conclude the section with a detailed example that satisfies the assumptions of this section. Assume that $X$ is a union of $n$ circles connected at one point. One can represent a function $f:X \to \mathbb{R}$ as a function 
\[
f=(f_1, \cdots, f_n)
\] 
where the $f_i:[0,1] \to \mathbb{R}$ are subject to the boundary conditions
\[
f_i(1)=f_i(0)=f_j(0)=f_j(1)\quad \text{for all $i\neq j$}.
\]
One considers then the Dirichlet form
\[
\mathcal{E} (f,g)=\sum_{i=1}^n \int_0^1 f'_i (x) g_i'(x) dx
\]
with domain the $\eng$-closure of 
\[
\{ f \in C^\infty( [0,1], \mathbb{R}^n ),  f_i(1)=f_i(0)=f_j(0)=f_j(1),~\forall i\neq j \}.
\]
 For every $f \in \dom \eng$, one has
\[
\sum_{i=1}^n \int_0^1 f'_i (x) dx=0,
\]
where the derivatives are understood in the distribution sense. Therefore  the reference  measure  $dx$ is the the energy measure of a harmonic form (namely, the energy measure of the differential form $\mathbf 1$ in the isometry described in proposition \ref{prop:MGIsometry}). For every  $f \in \dom \eng$, one has
\[
\star \partial f =(f'_1,\cdots,f'_n).
\]
One deduces that $\dom \partial^* \star$ is the $\eng$-closure of 
\[
\{ f \in C^\infty( [0,1], \mathbb{R}^n ), \sum_{i=1}^n f_i(0)=\sum_{i=1}^n f_i(1) \}
\]
and that for $f \in \dom \partial^* \star $,
\[
 \partial^* \star f=-(f'_1,\cdots,f'_n),
\]
where, once again, the derivatives are understood in the distribution sense.

\

We denote as before by $\Delta$ the generator of $\mathcal{E}$ and $P_t=e^{t \Delta}$. Denote now $P_t^S$ the standard heat semigroup on $[0,1]$ with periodic boundary condition and by $P_t^D$ the Dirichlet heat semigroup (zero boundary condition) on $[0,1]$. By extension for $f=(f_1,\cdots,f_n) \in L^2([0,1],dx)^n$, we denote
\[
P^S_tf=( P^S_tf_1,\cdots, P^S_tf_n),
\]
and we adopt a similar convention for $P_t^D$. The generators of $P_t^S$ and $P_t^D$ are respectively denoted by $\Delta^S$ and $\Delta^D$ and the corresponding Dirichlet form by $\mathcal{E}^S$ and $\mathcal{E}^D$. 
If we denote
\[
\mathfrak{L}=\{ f \in L^2(X) , f_1=\cdots=f_n \},
\]
any function $f$ can uniquely be decomposed as $f=f_\mathfrak{L} +f_\mathfrak{L^\perp}$ where $f_\mathfrak{L} \in \mathfrak{L}$ and $f_{\mathfrak{L}^\perp} \in \mathfrak{L}^\perp$. We have then the following proposition:

\begin{proposition}
\
\begin{enumerate}
\item Let $f \in L^2(X,dx)$, then $f \in \dom \Delta$ if and only if $f_{\mathfrak{L}} \in \dom \Delta^S$ and  $f_{\mathfrak{L}^\perp}$ is in $\dom \Delta^D$. In that case,
\[
\Delta f= \Delta^S f_{\mathfrak{L} }+ \Delta^D f_{\mathfrak{L}^\perp}.
\]
\item Let $f \in L^2(X,dx)$, then $f \in \dom \Delta^\perp$ if and only if $f_{\mathfrak{L}} \in \dom \Delta^S$ and $f_{\mathfrak{L}^\perp} \in \dom \Delta^N$. In that case,
\[
\Delta^\perp f= \Delta^S f_{\mathfrak{L}} + \Delta^N f_{\mathfrak{L}^\perp}.
\]
\end{enumerate}
\end{proposition}

\begin{proof}
If $ f , g\in C^\infty( [0,1], \mathbb{R}^n ),  f_i(1)=f_i(0)=f_j(0)=f_j(1),~\forall i\neq j$. Then $f_{\mathfrak{L}} , g_{\mathfrak{L}} \in \dom \eng^S $, $f_{\mathfrak{L}^\perp}, g_{\mathfrak{L}^\perp} \in \dom \eng^D$ and
\[
\eng(f,g)=\eng^S (f_{\mathfrak{L}} , g_{\mathfrak{L}})+ \eng^D(f_{\mathfrak{L}^\perp}, g_{\mathfrak{L}^\perp}).
\]
Thus
\begin{align*}
\begin{array}{lll}
\dom \eng & \to & \dom \eng^S \otimes \dom \eng^D \\
f & \to & (f_{\mathfrak{L}} , f_{\mathfrak{L}^\perp})
\end{array}
\end{align*}
is seen to be a Dirichlet space isomorphism and Part 1 follows. Part 2 follows from  the fact that $(\Delta^S)^\perp = \Delta^S$ and  $(\Delta^D)^\perp =\Delta^N$. 
\end{proof}

The next corollary easily follows and illustrates Theorem \ref{lkmn}.

\begin{corollary}
\

\begin{enumerate}
\item Let $f \in L^2(X,dx)$. Then for every $t \ge 0$,
\[
P_t f = P^S_tf_{\mathfrak{L}}+P_t^Df_{\mathfrak{L}^\perp}.
\]
\item Let $f \in L^2(X,dx)$. Then for every $t \ge 0$,
\[
P^\perp _t f = P^S_tf_{\mathfrak{L}}+P_t^Nf_{\mathfrak{L}^\perp}.
\]
\end{enumerate}
As a consequence
\[
\mathfrak{L}=\left\{ f \in L^2(X,\mu), \text{ for every } t \ge 0, e^{t\Delta } f =e^{t\Delta^\perp }f  \right\}.
\]
\end{corollary}

\section{Bakry-\'Emery estimate on metric graphs}

In this section  we prove the validity of the Bakry-\'Emery estimate on metric graphs with finite number of edges and rays. The results of Section 3  may therefore be applied in that class of examples.

\subsection{Function spaces and differential one-forms on metric graphs}\label{subsec:generalMG}

In Section 4 we developed a Poincar\'e duality based on a Hodge star operator when the reference measure is an energy form $\nu_\omega$ for some $\omega \in \irt$. This requires the total measure of the space to be finite, ruling therefore out non-compact metric graphs. Our first task will therefore be to find an isomorphism  between one-forms and functions that works for any metric graph. This will be made possible by the existence of the derivative operator.

\

For a reference on the general theory of metric graphs we refer to \cite{Pos12}. We start off with notations concerning (discrete) weighted graphs. We use $\cell$ to denote a graph, which is composed of verteces $V$, (internal) edges $\edge$ and rays $\mathsf R$. For each edge $e\in\edge$ there is two endpoints $e^-$ and $e^+$ in $V$ as well as a length $r(e)>0$. Rays have one associated endpoint $e^-$ in $V$ and the length is infinite.  For $v\in V$ define the set of adjacent edges $\edge_v = \set{e\in \edge\cup \mathsf R~|~v = e^- \text{ or } v = e^{+}}$. We assume that $\edge$ and $\mathsf R$ are finite.

\

Define $\gmet$ to be the metric graph associated with $\cell$: For $e\in\edge$ let $I_e = [0,r(e)]$ and if $e\in\mathsf R$ then $I_e = [0,\infty)$. In this case $\gmet$ is the set $\sqcup_{e\in\edge\cup\mathsf R} I_e$ modulo the equivalence relation which identifies  endpoints of $I_{e_1}$ and $I_{e_2}$ if associated endpoints of $e_1$ and $e_2$ are the same vertex. Define $\Phi_e: I_e\to\gmet$ to be the projection onto the equivalence classes. For example $\Phi_{e_1}(0) = \Phi_{e_2}(r(e_2))$ if $e^-_1 = e^+_2$. We may think of $I_e$ as subsets of $\gmet$ and refer to $0 \in I_e$ as $e^-$ and $r(e) \in I_e$ as $e^+$.

\

Now, we shall define some notations concerning the function spaces on $\gmet$. Define the reference measure $\mu$ on $\gmet$ to be that which is Lebesgue measure when restricted to each $I_e$. Functions $f\in L^2(\gmet) = L^2(\gmet,\mu)$ will be denoted as vectors $f = (f_e)_{e\in \edge\cup\mathsf R}$ where $f_e \in L^2(I_e)$, i.e. $L^2(\gmet) = \oplus_e L^2(I_e)$. Other function spaces have similar vector decompositions, perhaps with boundary conditions. For example, we shall think of  continuous functions $C(\gmet)$ to be the vectors with entries in $ C(I_e)$ where, if $v\in I_{e_1}$ and $I_{e_2}$ then $f_{e_1}(v) = f_{e_2}(v)$. Define the Sobolev space $H^1_0(\gmet)$ to be  the functions $f$ such that $f_e \in H^1(I_e)$, i.e. both $f_e$ and $f_e'$ are in $L^2(I_e)$, with the boundary conditions ensuring that $f$ is continuous at verteces.

\ 

When it is well defined, we consider $f(v)$ to be the vector $(f_e(v))_{e\in \edge_v}$ of values of $f$ (or traces of $f$) at the associated endpoint of $e$. We shall need to also denote the multiplication (diagonal) operator $U_v(e) = 1$ if $v = e^-$ and $U_v(e) = -1$ if $v = e^+$ for each $v\in V$. In this way, the inward facing normal derivative of $f$ at $v\in V$ along an edge $e$ is $U_v(e)f_e'(v)$. Here and later, $f_e'(0)$ or $f'_e(r(e))$ is understood to be the trace of $f'_e$ onto the boundary of $I_e$. 

\

One defines first  a  derivative operator $d: H^1_0(\gmet) \to L^2(\gmet)$ by $(df)_e(x) = f'_e(x)$, which we will concisely denote by $df = f'$. Note that, up to a sign, $d$ depends on the orientation of the graph. However, the Dirichlet form defined by
\[
\eng(f,g) = \int f'g' \ d\mu = \gen{f',g'}
\]
nor its generator $\Delta f = -f''$ depend on this orientation. The domain of $\eng$ is $H^1_0(\gmet)$ and the domain of $\Delta$ is 
\[
\dom\Delta = \set{f\in H^1_0(\gmet)~|~ \forall ~e, f_e' \in H^1(I_e), \forall ~v\in V,~\sum_{v\in \edge_v}U_v(e)f_e'(v)=0}.
\]
These boundary conditions are called standard or Kirchhoff boundary conditions. The carr\'e du champ  associated to $\eng$ or $\Delta$ is  $\Gamma(f,g) = f'g'$ for $f,g\in H^1_0(\gmet)$.

We also define the codifferential $d^* f  := -f'$ to be the adjoint of $d^*$. Using the integration by parts formula
\[
\gen{f',g} = -\gen{f,g'} + \sum_{v\in V}\sum_{e\in E_v} U_v(e)f_e(v)g_e(v),
\]
one sees that
\[
\dom d^* = H^1_1(\gmet) := \set{f\in L^2(\gmet)~|~f_e\in H^1(I_e),~\forall~v\in V,~\sum_{e\in \edge_v} U_v(e)f_e(v) = 0}. 
\]
The following result shows that we can identify the space of one-forms in the sense of Section 2 with $L^2(\gmet)$.

\begin{prop}\label{prop:MGIsometry}
Let $\gmet$ be a metric graph with a finite number of edges and rays, and $\eng$ be the Dirichlet form defined  above with $\dom\eng = H_0^1(\gmet)$. If $\irt$ be the space of differential 1-forms, then $\irt \cong L^2(\gmet)$ via an isometry which sends $\partial f \mapsto f'$ for all $f\in \dom\eng$. Under this isometry $\partial^* = d^*$, $\dom\partial^* = H^1_{1}(\gmet)$, and $\vec\Delta = dd^*$.
\end{prop}

\begin{proof}
This is an expansion of comments made in \cite[Section 5]{IRT12}, we include a quick argument for the sake of completeness. It is straightforward to see that,
\[
\norm{f\otimes g}_{\irt}^2 = \int g^2(f')^2 \ d\mu = \norm{gf'}_{L^2(\gmet)}^2
\]
and thus the function which maps $f\otimes g \to gf'$ is an isomorphism of simple tensors and thus extends to an isomorphism. Under this isomorphism, $\partial = d$ and hence $\partial^* = d^*$.
\end{proof}

In view of the previous isomorphism, we will simply denote $\vec \Delta f = dd^* f = -f''$. The domain is
\[
\dom\vec\Delta = \set{f\in H^1_1(\gmet)~|~f'\in H^1_0(\gmet)  }
\]
i.e. for each $v\in V$ $\sum_{e\in\edge_v} U_v(e)f_e(v) =0$ and for any pair of $e_1,e_2\in \edge_v$ then $f'_{e_1}(v) = f'_{e_2}(v)$. These are sometimes referred to as anti-Kirchhoff boundary conditions.

\begin{remark}

 A metric graph $\gmet$ admits a Poincar\'e duality in the sense of Section 4 if there is a function $h$  in $H^1_1(\gmet)$ with $|h|=1$ almost everywhere. i.e. $h=\pm1$ on each edge where the $\pm$ depends on the orientation of the edge. Alternatively, such a form exists, if there is an orientation such that $\sum_{e} \int_{0}^{r_e} f' \ dx  = \gen{f',h} = -\gen{f,h'}= 0$ for all $f\in H_0^1(\gmet)$. In the case that $\gmet$ admits a Poincar\'e duality, $\vec\Delta = \Delta^\perp$.
\end{remark}

\subsection{Heat Kernels and Bakry-\'Emery Estimates on Metric Graphs}

In this section, we present a formula for the kernel of the semigroups generated by $\Delta$ and $\vec \Delta$ as a sum over (combinatorial) paths. We assume, as before, that $\gmet$ is a metric graph with a finite number of edges and rays, and that $\gmet$ has no tadpoles --- that is edges $e$ such that $e^+ = e^-$. This assumption does not limit the metric spaces which the following discussion applies to: one can introduce a vertex at the midpoint of any tadpole, producing a metric graph which is isometric (as a metric space) to the original space. 

\

A combinatorial path $c$ from $x\in e_0$ to $y\in e_{n+1}$ is the $2n+1$-tuple $$(e_0,v_0,e_1,v_1,\ldots,v_n,e_{n+1}),$$ where  for $k=0,1,2,\ldots,n$, $v_k$ and $v_{k+1}$ are distinct endpoints of $e_k$.
 Without loss of generality we assume that $v_0 = e_0^+ = \Phi_{e_0}(r(e_0))$ and $v_n = e_{n}^- = \Phi_{e_n}(0)$. 
 
 \
 
 We can define two distinct notions of the length of a path $c$: the combinatorial length, which will be denoted $|c|$ and is $n+1$ (the number of verteces $c$ passes through) and the metric length (or simply length) 
\[
d_c(x,y) := \abs{r(e)-x} + \abs{y} + \sum_{k=1}^n r(e_k) .
\]
This is the length of the shortest path which follows the combinatorial path from $x$ to $y$, and hence depends on the endpoints as well as $c$.  

Using the work of \cite{FOT11,Sto10}, we observe that the natural distance
\[
\rho(x,y) = \sup\set{|f(x)-f(y)|~:~ \Gamma(f,f) = |f'|^2\leq 1}
\] 
coincides with the natural length metric on the space.

Define $C(x,y)$ to be the set of the combinatorial paths connecting $x$ to $y$, including, if $x$ and $y$ are in $e_0$, then the trivial path $(e_0)$, defining $S((e_0)) = 1$ and $d_{(e_0)}(x,y) = |x-y|$. Define the scattering amplitude associated to a combinatorial path
\begin{align}\label{eqn:PathScattering}
S(c) = \prod_{k=0}^n \paren{\frac{2}{\deg_{v_k}} -\delta_{e_k,e_{k+1}}}
\end{align}
where $\deg_v$ is the vertex degree of $v$ and $\delta_{e_k,e_{k+1}}$ is the Dirac Delta (i.e. 1 if $e_k = e_{k+1}$ and $0$ otherwise).

Letting $g_t(u):= (4\pi t)^{-1/2}e^{-u^2/4t}$, according to the formula in \cite[Corollary 3.4]{KPS07}, the heat kernel of $\Delta$ has the form

\begin{align}\label{eqn:HK}
 p_t(x,y) =  \sum_{c\in C(x,y)} S(c)g_t(d_c(x,y)).
\end{align}

\begin{prop}
The integral kernel of the semigroup associated to the anti-Kirchhoff Laplacian $\vec\Delta$ is
\[
\vec p_t (x,y) = \sum_{c\in C(x,y)} \vec S(c) g_t(d_c(x,y)),
\]
where $\vec S(c)$ is the anti-Kirchhoff scattering amplitude defined 
\begin{align}\label{eqn:AntiPathScattering}
\vec S(c) = \prod_{k=0}^n U_v(e_k)U_v(e_{k+1})\paren{\frac{2}{\deg_{v_k}} -\delta_{e_k,e_{k+1}}} = \pm S(c).
\end{align}
\end{prop}
\begin{proof}
To apply \cite[Corollary 3.4]{KPS07}, we need to verify the technical condition  that anti-Kirchhoff boundary conditions correspond to a maximal isotropic subspace in the sense of \cite{KPS07}. Following example 2.8 of \cite{KPS07}, the anti-Kirchoff vertex space at each vertex is that of Kirchhoff vertex space multiplied by the diagonal matrix $-U_v$, thus the conclusion. Alternatively, from remark 5.8 in \cite{Pos09}, the maximal isotropic condition is equivalent to the associated Laplacian being self-adjoint and we know that $\vec \Delta = dd^*$ is self-adjoint.
\end{proof}

\begin{example}[Walsh spider]
One can illustrate the previous formulas in the case of the Walsh spider. The Walsh spider with $N$ legs is the graph consisting on $N$ copies of $[0,+\infty)$ which we shall call $\set{I_j}_{j=1}^N$ identified at the respective 0. Calculating from the formula \eqref{eqn:HK} or using \cite{BPY}, one sees that the heat kernel has the form
\[
p_t(x_j,y_k) = \begin{cases}
\displaystyle \frac2N \frac{e^{-\abs{x_j+y_k}^2/4t}}{\sqrt{4\pi t}} & \text{if } j\neq k \\
 \displaystyle \frac1{\sqrt{4\pi t}}\paren{e^{-\abs{x_j - y_k}^2/4t} - \paren{1-\frac2N} e^{-\abs{x_j+y_k}^2/4t}} & \text{if }j=k.
\end{cases}
\]
where $x_i \in I_i$ and $y_k \in I_k$.  It follows that, if $ \vec p_t$ is the integral kernel of $\vec{\Delta}$, then 
\[
\vec p_t(x_j,y_k) = \begin{cases}
\displaystyle -\frac2N \frac{e^{-\abs{x_j+y_k}^2/4t}}{\sqrt{4\pi t}} & \text{if } j\neq k \\
 \displaystyle \frac1{\sqrt{4\pi t}}\paren{e^{-\abs{x_j - y_k}^2/4t} + \paren{1-\frac2N} e^{-\abs{x_j+y_k}^2/4t}} & \text{if }j=k.
\end{cases}
\]
Observe that this kernel takes values which are both positive and negative.  From this one sees that the ratio
\[
\frac{p_t(x_j,y_k)}{\abs{ p^\perp_t(x_j,y_k}} = \begin{cases}
1 & \text{if } j\neq k \\
\displaystyle  \frac{ 1- K e^{-x_jy_k/2t}}{1+K e^{-x_jy_k/2t}} & \text{if } j=k.
\end{cases}
\]
where $K=(1-2/N)$. It is easy to see that the above ratio is bounded between 1 and $(1-K)/(1+K) = 1/(N-1)$.  Integrating we get the inequality $|e^{t \Delta^\perp} f|(x) \leq (N-1) e^{t\Delta}|f|(x)$, which implies that the following Bakry-\'Emery estimate holds on the Walsh spider
\[
\sqrt{\Gamma(e^{t\Delta} f)}(x) \leq (N-1) e^{t\Delta}\sqrt{\Gamma(f)}(x).
\]
Observe that the constant $N-1$ is optimal in the previous estimate. Indeed, in the Walsh spider, the range of $d$ is dense in $L^2$, as a consequence the inequality 
\[
\sqrt{\Gamma(e^{t\Delta} f)}(x) \leq C e^{t\Delta}\sqrt{\Gamma(f)}(x),\quad f \in \dom \eng .
\]
is equivalent to the inequality 
\[
|e^{t \vec \Delta} f|(x) \leq C e^{t\Delta}|f|(x), \quad f \in L^2(X),
\]
which is equivalent to the bound $| \vec p_t (x,y)| \le C p_t(x,y)$.

\end{example}

With this example in mind, we now return to the study of general graphs.

\begin{lem}\label{estimates graph}
Assume that $\gmet$ has a finite number of edges and rays. For $T>0$, there exists a constant $C_1>0$   that depends only on $T$ and the graph $\gmet$, such that for $0<t\leq T$ and $\mu$ almost every $x,y$ 
\[
\abs{\vec p_t(x,y)} \leq C_1 g_t(\rho(x,y)),\quad\quad p_t(x,y) \leq C_1 g_t(\rho(x,y))
\]
Further, there exists a $T_0$ such that for all $0<t<T_0$, and $\mu$ almost every $x,y$
\[
p_t(x,y)\geq C_0g_t(\rho(x,y)).
\]
Here $C_0$ and $T_0$ only depend on the geometry of $\gmet$ ---  on the maximum vertex degree, the minimum edge length and the number of internal edges.
\end{lem}

\begin{remark}
The absolute values around $\vec p_t$ are important because it may be negative, as is the case in the case of the Walsh spider studied in the previous example.
\end{remark}

\begin{proof}
{\bf Upper bound.} First, since $\rho(x,y) = \inf_{c\in C(x,y)}d_c(x,y)$, for a fixed $a > 0$ and any $x,y$ there is a bounded number of paths $c\in C(x,y)$ such that $d_c(x,y) \leq \rho(x,y) + a$. To see this we may assume that $x,y$ are both in finite length edges. This is because a combinatorial path to/from a point on a ray is determined by the path taken until the last time it crosses the 0 of that ray. So either $x$ and $y$ are in internal edges, or we can replace them with the endpoints of the ray they are in.

 For any $x$ in an internal edge, the number of paths starting from $x$ and of length bounded by $M>0$ is less than $(\deg_{max}+1)^{M/r_{min}}$ where $\deg_{max}$ is the maximum vertex degree and $r_{min}$ is the minimum edge length. Thus there is an upper bound independent of our choice of $x$. The claim follows because the interior of the graph is compact. Further, if we take 
\[
\diam = \sup\set{\rho(x,y) ~|~\exists ~e_1,e_2\in\edge,~x\in I_{e_1},~y\in I_{e_2}}
\]
to be the farthest apart two points on finite length edges of $\graph$ can be, then
the number of paths from $x$ to any point $y$ of length less than $\rho(x,y) +a$ is bounded by $(\deg_{max}+1)^{(\diam+a)/r_{min}}$.

Because 
\[
g_t(d_c(x,y))/g_t(\rho(x,y)) = \exp\paren{-\frac{d_c(x,y)^2 - \rho(x,y)^2}{4t}},
\]
both $\abs{\vec p_t}$ and $\abs{p_t} = p_t$ are bounded above by
\[
g_t(\rho(x,y))\sum_{c\in C(x,y)} \abs{S(c)} \exp\paren{-\frac{d_c(x,y)^2 - \rho(x,y)^2}{4t}}
\] 
Factoring out the $g_t(\rho(x,y))$, we break the sum
\[
\sum_{c\in C(x,y)} \abs{S(c)} \exp\paren{-\frac{d_c(x,y)^2 - \rho(x,y)^2}{4t}} = A_I + A_{II}
\]
 into parts $A_I$ and $A_{II}$. Here $A_I$ is the sum over $c\in C(x,y)$ with $d_c(x,y) \leq \rho(x,y) + a$, and $A_{II}$ is the sum over $c\in C(x,y)$ with $d_c(x,y) > \rho(x,y) + a$.  Then
\[
A_I \leq \sum_{c\in C(x,y), d_c(x,y) \leq \rho(x,y) + a} \abs{S(c)} \leq (\deg_{max}+1)^{\frac{\diam + a}{r_{min}}},
\]
and, using an argument similar to the proof of Lemma 3.2 in \cite{KPS07}, one sees that 
\begin{align*}
A_{II} &=\sum_{c\in C(x,y), d_c(x,y) > \rho(x,y) + a} \abs{S(c)} \exp\paren{-\frac{(d_c(x,y)+\rho(x,y))(d_c(x,y) - \rho(x,y))}{4t}}\\
& \leq \sum_{c\in C(x,y), d_c(x,y) > \rho(x,y) + a}  \exp\paren{-\frac{ar_{min} |c|}{4t}}\\
& =  \sum_{n=1}^\infty \sum_{|c| = n} \exp\paren{-\frac{ar_{min} |c|}{4t}}\\
&\leq \sum_{n=1}^\infty \abs{\mathsf E}^n \exp\paren{{-\frac{a r_{min} n}{4t}}}
\end{align*}
because $d_c(x,y) + \rho(x,y) > d_c(x,y) \geq r_{min}|c|$, $d_c(x,y) - \rho(x,y) > a$, $|S(c)| \leq 1$ and the number of paths $c\in C(x,y)$ of combinatorial length $|c| = n$ is less than $|\mathsf E|^n$. Here $|\mathsf E|$ is the number of finite length edges of $G$.

For $t$ small enough
\[
\sum_{n=1}^\infty \abs{\mathsf E}^n \exp\paren{{-\frac{a r_{min} n}{4t}}}=\frac{\abs{\mathsf E}e^{-ar_{min}/4t}}{1- \abs{\mathsf E}e^{-ar_{min}/4t}}, 
\]
and choosing $a$ large, we can show this is bounded in the interval $(0,T)$ for any $T>0$.

{\bf Lower Bound.} Note that if $c_0$ is such that $d_{c_0}(x,y) = \rho(x,y)$, then $0< S(c_0)$ because, with Kirchhoff conditions any negative terms in the product that make up $S(c_0)$ would come from a combinatorial path which has two consecutive edges which are the same, in which case, a shorter combinatorial path $c'$ could be constructed by removing this sequence of two edges.

If $x,y$ are not in the same edge, then the combinatorial path $c$ must visit 2 vertices for any $c$ with $S(c)<0$, and in this case, this implies that, using the notation from the previous paragraph that $d_{c_0}(x,y) + r_{min} \leq d_{c}(x,y)$. Thus, if $x,y$ are not on the same edge, and $d_c(x,y) - \rho(x,y) < r_{min}$ then $S(c) \geq 0$. Hence, setting the $a$ above to be $r_{min}$,
\[
p_t(x,y) \geq g_t(d_{c_0}(x,y))\paren{A_I - A_{II}}.
\]
Since there is at least one path from $x,y$ with $d_{c_0}(x,y)$,
\[
A_I \geq S(c_0) > \paren{\frac2{\deg_{max}}}^{\diam/r_{min}}
\]
and since $A_{II} \to 0$ at $t\to 0$, then we can find $T_0$ such that the lower bound holds.

If $x$ and $y$ are in the same edge $e$, and $e$ has vertices $v_-$ and $v_+$ then, choosing $a< r_{min}$ implies that the sum becomes, if $x$ and $y$ are in an internal edge
\begin{align*}
p_t(x,y) & = g_t(|x-y|) + \frac{2-d_{v_-}}{d_{v_-}} g_t(x+y) +  \frac{2-d_{v_+}}{d_{v_+}} g_t(2r_e-x-y) + \sum_{c: d_c(x,y) > \rho(x,y) + a} S(c) g_t(d_c(x,y)) \\
& \geq g_t(|x-y|) \paren{\frac13 - A_{II}}. 
\end{align*}
If $x,y$ are in  the same external edge, a slight modification above shows that $p_t(x,y) \geq g_t(|x-y|)(\frac23 -A_{II})$.
\end{proof}

We are now ready to prove the main result of the section.

\begin{thm}
Assume that $\gmet$ has a finite number of edges.

\begin{enumerate}
\item If $\gmet$ is compact, then there exist a constant $C>1$ and a constant $K>0$ such that for every $ f \in \dom \eng$ and $t \ge 0$,
\[
\sqrt{\Gamma (P_t f)} \le C e^{-Kt} P_t \sqrt{\Gamma ( f)}.
\]
\item  If $\gmet$ is not compact, then there exist a constant $C>1$ and and a constant $K \ge 0$ such that for every $ f \in \dom \eng$ and $t \ge 0$
\[
\sqrt{\Gamma (P_t f)} \le C e^{Kt} P_t \sqrt{\Gamma ( f)}.
\]
\end{enumerate}
\end{thm}

\begin{proof}
Since  $\gmet$ has a finite number of edges, as a consequence of Lemma \ref{estimates graph}, we deduce that there exists a constant $C >1$ such that for $0 <t \le 1$,
\[
\frac{\abs{\vec p_t(x,y)}}{p_t (x,y)} \le C.
\]
From Theorem \ref{intertwining1},  for $f \in \dom \eng$,
\[
  \partial e^{t\Delta} f=e^{t\vec{\Delta}}  \partial f, \quad t \ge 0.
\]
Thus, for $0 \le t \le 1$, we have
\begin{align}\label{BE proof}
\sqrt{\Gamma (P_t f)} \le C  P_t \sqrt{\Gamma ( f)}.
\end{align}
We now discuss the two cases:
\begin{enumerate}
\item \textbf{$\graph$ is compact.} In that case  $\Delta$ has a pure point spectrum, $ 1 \in \dom \Delta$ and the Dirichlet space $(\eng , \dom \eng)$ satisfies a Poincar\'e inequality:
\[
\int_X\left( f-\int_X f d\mu\right)^2 d\mu \le \frac{1}{\lambda_1} \eng (f,f), \quad f \in \dom \in  \eng.
\]
Moreover, it is easy to check that  there exists a $M>0$ such that  for $\mu$-almost every $x,y \in X$
\[
p_{1}(x,y) \le M, \quad |   \Gamma(p_{1} (.,y))(x) | \le M.
\]
See \eqref{plok} for the bound on $ \Gamma(p_{1} (.,y))(x) $. We conclude then as a consequence of Theorem \ref{BE large time}.
\item \textbf{$\graph$ is not compact.}  One can use Theorem \ref{self_gh}.

\end{enumerate}
\end{proof}

We can give a lower bound estimate on the optimal constant in the inequality
\[
\sqrt{\Gamma (P_t f)} \le C e^{Kt} P_t \sqrt{\Gamma ( f)}.
\]

\begin{thm}\label{thm:LowerBoundOnConstant}
Assume that $\graph$ has a finite number of edges. Let $\tau >0$. The optimal constant $C$ in the inequality
\[
\sqrt{\Gamma (P_t f)} \le C  P_t \sqrt{\Gamma ( f)}, \quad f \in \dom \eng, 0 \le t \le \tau
\]
satisfies
\[
C \ge  \max (\deg v -1)
\]
where the maximum is taken over the set of vertices of $\graph$. 
\end{thm}

\begin{proof}
The idea is  to use a local comparison to the Walsh spider around vertexes and a scaling argument. Let $v$ be a vertex in $\graph$. For $c>0$, we denote by $\graph^c$ the metric graph obtained from $\graph$ by multiplying all distances by $c$. Denote by $\delta_c : \graph \to \graph^c$ the dilation that fixes $v$. Let now $X$ be the Walsh spider with $N$ legs where $N=\deg v$. A function $f=(f_1,\cdots,f_N) \in L^2(X)$ defines a function $\tilde f$ on the graph $\graph^c$  by identifying $v$ with the center of the Walsh spider, numbering the edges adjacent to $v$ and defining $\tilde{f} (x_i)=f_i ( d(v,x_i))$ when $x_i$ is in the edge numbered $i$ and $\tilde{f}=0$ on edges which are not adjacent to $v$. When $c \to +\infty$, one has 
\[
P_{t /c^2} ( \tilde f \circ \delta_c)( \delta_c^{-1} x_i) \to (P_t^X f)(x^*_i),
\]
where $x^*_i \in X$ is the point on the leg $i$ such that $x^*_i=d(v,x_i)$. Rescaling then the inequality
\[
\sqrt{\Gamma (P_t \tilde f)} \le C  P_t \sqrt{\Gamma ( \tilde f)}, \quad 0 \le t \le \tau
\]
and taking the limit when $c \to +\infty$ yields
\[
\sqrt{\Gamma^X (P^X_t  f)} \le C  P^X_t \sqrt{\Gamma^X (  f)}, \quad t \ge 0.
\]
Since it is true for every $f$, one must have $ C \ge \deg v -1$.
\end{proof}

\subsection{Local Riesz transform on non-compact metric graphs}

For this subsection we assume that $\gmet$ is a non-compact metric graph with a finite number of edges and rays. We prove the following theorem. We wish to use results from \cite{ACDH04}, so we first need to establish that the current setting matches that in the article. In particular, we follow the checklist indicated on page 922 in the local form. For all $t\geq 0$,
\[
\frac{t}2 \leq \mu(B_t(x)) \leq Ct
\] 
where $C$ is bounded by the number of edges. This is stronger than volume doubling.  Doubling is important in the proofs, because it allows us to use the Hardy--Littlewood maximal operators on $\graph$ as indicated in \cite{Hei01}. Also, the lower bound above does not hold for all times if $\graph$ is compact. 

 It is established in \cite{Hae11} that these metric graphs satisfy Gaussian Heat Kernel estimates and Poincar\'e inequality (alternatively, earlier in this section we established local upper Gaussian estimates, which are sufficient for our situation). It is also well established that $P_t$ is conservative, i.e. $\int p_t(x,y) \ d\mu(x) = 1$. 
 Further, from the standard theory of Dirichlet forms $\norm{(-\Delta)^{1/2} f}^2 = \eng(f) = \norm{df}^2$, and thus the Riesz transform is $L^2$ bounded.
The Laplacian operator is elliptic, by virtue of the fact that it looks like the 1-dimensional Laplacian almost everywhere.

\begin{theorem}
There is $\alpha>0$ such that for all $a \ge \alpha $, the local Riesz transform $d(-\Delta + a)^{-1/2}$ is bounded in $L^p$ for all $p$ with $1<p<\infty$.
\end{theorem}

\begin{proof}

The proof leverages \cite[Theorem 1.8]{ACDH04} which states that, for a metric measure space with local volume doubling (in our case, implied because there are a finite number of edges) and local upper estimates on the diagonal of the heat kernel (in our case, implied by lemma \ref{estimates graph}), if there is $\beta>0$ such that 
\begin{align}\label{plok}
|d_x p_t(x,y)| \leq \frac{C e^{\beta t}}{\sqrt{t}\mu(B_{\sqrt{t}}(x))} 
\end{align}
 then the local Riesz transform is bounded. This heat kernel gradient estimate is the condition referred to as $G_{loc}$ in \cite{ACDH04}.

Calculating from above
\begin{align*}
\abs{d_xp_t(x,y)} & = \abs{ \sum_{c\in C(x,y)}  S(c)\frac{\pm d_c(x,y)}{2t}g_t(d_c(x,y))}\\
            & \leq \frac{1}{t \sqrt{4\pi}}\sum_{n=0}^\infty |\edge|^n \frac{nr_{min}}{t^{1/2}} e^{-r^2_{min} n^2/(4t)} \\ 
            & \leq \frac{1}{t\sqrt{4\pi}}\sum_{n=0}^\infty |\edge|^n e^{-r^2_{min} n^2/8t}.
\end{align*}
Where $r_{min}$ is the minimum length of an edge, the second inequality is by the same argument from Lemma 3.2 in \cite{KPS07}(as we used before), and the last inequality is because $xe^{-x^2} \leq e^{-x^2/2}$. Letting $L = \log(|\edge|)$,
\[
|\edge|^n e^{-n^2r_{min}^2/8t} = e^{Ln-n^2r_{min}^2/8t} \leq e^{-2Ln + 8L r_{min}^{-2} t}
\]
by taking the Taylor expansion of $xL - \frac{x^2r_{min}^2}{8t}$ around $x = 8tLr^{-2}_{min}$.

The above sum is thus bounded, and we get
\[
|d_xp_t(x,y)| \leq \frac{1}{t\sqrt{4\pi}}\frac{e^{8Lr_{min}^{-2} t}}{1-e^{-L}} \leq \frac{e^{8Lr_{min}^{-2}t}}{\sqrt{t}V(x,\sqrt{t})} \frac{1}{\sqrt{\pi}(1+|\edge|)}.
\]

\end{proof}

\subsection{Invalidity of Ricci Curvature lower bounds}

In this section we point out that no metric graph with standard boundary conditions and a vertex with degree more than two can satisfy the Ricci Curvature lower bounds of Sturm--Lott--Villani, which shall be denoted $CD(K,\infty)$ for any $K$. This is obviously not surprising since, from recent works (see  \cite{AGS14} and \cite{AGS15}), under suitable assumptions a generalized Ricci curvature lower bound is actually equivalent to a classical Bakry-\'Emery estimate:
\[
\Gamma (e^{t \Delta} f) \le e^{2Kt} e^{t \Delta} \Gamma(f).
\]
Let the set
\[
[A,B]_t = \set{z\in A ~|~ \exists x\in A,~y\in B \text{ such that } d(x,z) =  t d(x,y) \text{ and } d(z,y) = (1-t) d(x,y)}
\]
for $t\in [0,1]$ is the set of points which are convex combinations of $A$ and $B$ in that the lie on a geodesic connecting a point $x$ in $A$ to a point $y$ in $B$ at the portion $t$ along the curve. The idea is to prove the invalidity of the Brunn-Minkowski inequality.

Let $W$ denote the Wasserstein distance function on probability measures on a geodesic metric measure space $(X,d, \mu)$. We shall need no properties of the Wasserstein distance other than the fact that it is a metric on probability measures on a metric space, and hence is positive for two different measures.

The Brunn-Minskowski inequality refers to the following convexity condition
\[
\log(\mu([A,B]_t)) \geq t\log(\mu(A)) + (1-t)\log(\mu(B)) + \frac12 K t (1-t) W\paren{\frac{1_A}{\mu(A)} \mu,\frac{1_B}{\mu(B)}\mu}^2.
\]
It is proven in \cite[Proposition 2,1]{St2} that if a metric measures space $(X,d,\mu)$ which satisfies $CD(K,\infty)$, then for all sets $A$, $B$ and times $t\in[0,1]$, the above inequality holds. Showing that this inequality doesn't hold was used in \cite[Section 8.2]{Kaj13} to prove that $CD(K,\infty)$ does not hold for any $K$ on the harmonic Sierpinski gasket.

The intuitive reasoning why this inequality does not hold on metric graphs is that at each vertex with degree at least 3,  geodesics branch off from one another.

\begin{thm}\label{thm:BrunnMinkowskiFail}
Let $\graph$ be a metric graph with standard boundary conditions and $d$ is the intrinsic (geodesic) distance function on $\graph$, and let $\graph$ has at least one vertex with degree greater than 2. There are sets $A$ and $B$ in $\graph$ for which the Brunn--Minkowski inequality does not hold. Hence it is not possible for $\graph$ to satisfy $CD(K,\infty)$
for any $K$.
\end{thm}

\begin{proof}
We shall prove the inequality is not satisfied for the Walsh spider with three legs, $E_i = [0,\infty)$ for $i=0,1,2$. This can be generalized to any metric graph by considering a small neighborhood of a vertex with degree at least $3$.
Let $A  = (a_1,a_2) \subset E_0$ with $a_2- a_1 = \ell$. Let $B$ consist of two intervals $(b_1,b_2)$ contained in $E_1$ and $E_2$ with $b_2-b_1 = \ell$. Then, for $t$ close enough to 1, 
\[
[A,B]_t = (ta_1 - (1-t)b_2, ta_2 -(1-t)b_1) \subset E_0
\]
and hence
\[
\mu([A,B]) = t(a_2-a_1) + (1-t)(b_2 - b_1) =  \ell.
\]
On the other hand
\[
t\log(\mu(A)) + (1-t)\log(\mu(B)) = \log(\ell) + (1-t)\ln(2)\geq \log(\ell) = \log(\mu([A,B]_t)).
\]
Thus it is impossible to satisfy the Brunn--Minkowski inequality.
\end{proof}

%
%
\bibliography{Fractals}{}

\newcommand{\etalchar}[1]{$^{#1}$}
\def\cprime{$'$}
\providecommand{\bysame}{\leavevmode\hbox to3em{\hrulefill}\thinspace}
\providecommand{\MR}{\relax\ifhmode\unskip\space\fi MR }
\providecommand{\MRhref}[2]{%
  \href{http://www.ams.org/mathscinet-getitem?mr=#1}{#2}
}
\providecommand{\href}[2]{#2}
\begin{thebibliography}{ACDH04}

\bibitem[ABC{\etalchar{+}}00]{MR1845806}
C{\'e}cile An{\'e}, S{\'e}bastien Blach{\`e}re, Djalil Chafa{\"{\i}}, Pierre
  Foug{\`e}res, Ivan Gentil, Florent Malrieu, Cyril Roberto, and Gr{\'e}gory
  Scheffer, \emph{Sur les in\'egalit\'es de {S}obolev logarithmiques},
  Panoramas et Synth\`eses [Panoramas and Syntheses], vol.~10, Soci\'et\'e
  Math\'ematique de France, Paris, 2000, With a preface by Dominique Bakry and
  Michel Ledoux. \MR{1845806}

\bibitem[ACDH04]{ACDH04}
Pascal Auscher, Thierry Coulhon, Xuan~Thinh Duong, and Steve Hofmann,
  \emph{Riesz transform on manifolds and heat kernel regularity}, Ann. Sci.
  \'Ecole Norm. Sup. (4) \textbf{37} (2004), no.~6, 911--957. \MR{2119242
  (2005k:58043)}

\bibitem[AGS14a]{Amb}
Luigi Ambrosio, Nicola Gigli, and Giuseppe Savar{\'e}, \emph{Calculus and heat
  flow in metric measure spaces and applications to spaces with {R}icci bounds
  from below}, Invent. Math. \textbf{195} (2014), no.~2, 289--391. \MR{3152751}

\bibitem[AGS14b]{AGS14}
\bysame, \emph{Metric measure spaces with {R}iemannian {R}icci curvature
  bounded from below}, Duke Math. J. \textbf{163} (2014), no.~7, 1405--1490.
  \MR{3205729}

\bibitem[AGS15]{AGS15}
\bysame, \emph{Bakry-\'{E}mery curvature-dimension condition and {R}iemannian
  {R}icci curvature bounds}, Ann. Probab. \textbf{43} (2015), no.~1, 339--404.
  \MR{3298475}

\bibitem[Aid98]{MR1641566}
Shigeki Aida, \emph{Uniform positivity improving property, {S}obolev
  inequalities, and spectral gaps}, J. Funct. Anal. \textbf{158} (1998), no.~1,
  152--185. \MR{1641566}

\bibitem[BB16]{BBnon}
Fabrice Baudoin and Michel Bonnefont, \emph{Reverse {P}oincar\'e inequalities,
  isoperimetry, and {R}iesz transforms in {C}arnot groups}, Nonlinear Anal.
  \textbf{131} (2016), 48--59. \MR{3427969}

\bibitem[BBBC08]{BBBC}
Dominique Bakry, Fabrice Baudoin, Michel Bonnefont, and Djalil Chafa{\"{\i}},
  \emph{On gradient bounds for the heat kernel on the {H}eisenberg group}, J.
  Funct. Anal. \textbf{255} (2008), no.~8, 1905--1938. \MR{2462581}

\bibitem[BGL14]{BGL}
Dominique Bakry, Ivan Gentil, and Michel Ledoux, \emph{Analysis and geometry of
  {M}arkov diffusion operators}, Grundlehren der Mathematischen Wissenschaften
  [Fundamental Principles of Mathematical Sciences], vol. 348, Springer, Cham,
  2014. \MR{3155209}

\bibitem[BH91]{BH91}
Nicolas Bouleau and Francis Hirsch, \emph{Dirichlet forms and analysis on
  {W}iener space}, de Gruyter Studies in Mathematics, vol.~14, Walter de
  Gruyter \& Co., Berlin, 1991. \MR{1133391 (93e:60107)}

\bibitem[BPY89]{BPY}
Martin Barlow, Jim Pitman, and Marc Yor, \emph{On {W}alsh's {B}rownian
  motions}, S\'eminaire de {P}robabilit\'es, {XXIII}, Lecture Notes in Math.,
  vol. 1372, Springer, Berlin, 1989, pp.~275--293. \MR{1022917}

\bibitem[Bus82]{Bu}
Peter Buser, \emph{A note on the isoperimetric constant}, Ann. Sci. \'Ecole
  Norm. Sup. (4) \textbf{15} (1982), no.~2, 213--230. \MR{683635}

\bibitem[Che70]{Ch}
Jeff Cheeger, \emph{A lower bound for the smallest eigenvalue of the
  {L}aplacian}, Problems in analysis ({P}apers dedicated to {S}alomon
  {B}ochner, 1969), Princeton Univ. Press, Princeton, N. J., 1970,
  pp.~195--199. \MR{0402831}

\bibitem[CS03]{CS03}
Fabio Cipriani and Jean-Luc Sauvageot, \emph{Derivations as square roots of
  {D}irichlet forms}, J. Funct. Anal. \textbf{201} (2003), no.~1, 78--120.
  \MR{1986156}

\bibitem[Dav89]{Da}
E.~B. Davies, \emph{Heat kernels and spectral theory}, Cambridge Tracts in
  Mathematics, vol.~92, Cambridge University Press, Cambridge, 1989.
  \MR{990239}

\bibitem[DG54]{DG}
Ennio De~Giorgi, \emph{Su una teoria generale della misura
  {$(r-1)$}-dimensionale in uno spazio ad {$r$} dimensioni}, Ann. Mat. Pura
  Appl. (4) \textbf{36} (1954), 191--213. \MR{0062214}

\bibitem[Ebe99]{Ebe99}
Andreas Eberle, \emph{Uniqueness and non-uniqueness of semigroups generated by
  singular diffusion operators}, Lecture Notes in Mathematics, vol. 1718,
  Springer-Verlag, Berlin, 1999. \MR{1734956}

\bibitem[Eld10]{MR2557945}
Nathaniel Eldredge, \emph{Gradient estimates for the subelliptic heat kernel on
  {$H$}-type groups}, J. Funct. Anal. \textbf{258} (2010), no.~2, 504--533.
  \MR{2557945}

\bibitem[FOT11]{FOT11}
Masatoshi Fukushima, Yoichi Oshima, and Masayoshi Takeda, \emph{Dirichlet forms
  and symmetric {M}arkov processes}, extended ed., de Gruyter Studies in
  Mathematics, vol.~19, Walter de Gruyter \& Co., Berlin, 2011. \MR{2778606
  (2011k:60249)}

\bibitem[{Hae}11]{Hae11}
S.~{Haeseler}, \emph{{Heat kernel estimates and related inequalities on metric
  graphs}}, ArXiv e-prints (2011), arXiv 1101.3010.

\bibitem[Hei01]{Hei01}
Juha Heinonen, \emph{Lectures on analysis on metric spaces}, Universitext,
  Springer-Verlag, New York, 2001. \MR{1800917 (2002c:30028)}

\bibitem[Hin10]{Hin10}
Masanori Hino, \emph{Energy measures and indices of {D}irichlet forms, with
  applications to derivatives on some fractals}, Proc. Lond. Math. Soc. (3)
  \textbf{100} (2010), no.~1, 269--302. \MR{2578475 (2010k:60272)}

\bibitem[Hin13]{Hin13}
\bysame, \emph{Measurable {R}iemannian structures associated with strong local
  {D}irichlet forms}, Math. Nachr. \textbf{286} (2013), no.~14-15, 1466--1478.
  \MR{3119694}

\bibitem[HKT15]{HKT13}
Michael Hinz, Daniel~J. Kelleher, and Alexander Teplyaev, \emph{Metrics and
  spectral triples for {D}irichlet and resistance forms}, J. Noncommut. Geom.
  \textbf{9} (2015), no.~2, 359--390. \MR{3359015}

\bibitem[HRT13]{HRT}
Michael Hinz, Michael R{\"o}ckner, and Alexander Teplyaev, \emph{Vector
  analysis for {D}irichlet forms and quasilinear {PDE} and {SPDE} on metric
  measure spaces}, Stochastic Process. Appl. \textbf{123} (2013), no.~12,
  4373--4406. \MR{3096357}

\bibitem[HT14]{HT14}
M.~{Hinz} and A.~{Teplyaev}, \emph{{Local Dirichlet forms, Hodge theory, and
  the Navier-Stokes equations on topologically one-dimensional fractals}},
  Trans. Amer. Math. Soc. (2014).

\bibitem[IRT12]{IRT12}
Marius Ionescu, Luke~G. Rogers, and Alexander Teplyaev, \emph{Derivations and
  {D}irichlet forms on fractals}, J. Funct. Anal. \textbf{263} (2012), no.~8,
  2141--2169. \MR{2964679}

\bibitem[Kaj13]{Kaj13}
Naotaka Kajino, \emph{Analysis and geometry of the measurable {R}iemannian
  structure on the {S}ierpi\'nski gasket}, Fractal geometry and dynamical
  systems in pure and applied mathematics. {I}. {F}ractals in pure mathematics,
  Contemp. Math., vol. 600, Amer. Math. Soc., Providence, RI, 2013,
  pp.~91--133. \MR{3203400}

\bibitem[KM07]{MR2320383}
Hiroshi Kawabi and Tomohiro Miyokawa, \emph{The {L}ittlewood-{P}aley-{S}tein
  inequality for diffusion processes on general metric spaces}, J. Math. Sci.
  Univ. Tokyo \textbf{14} (2007), no.~1, 1--30. \MR{2320383}

\bibitem[KPS07]{KPS07}
Vadim Kostrykin, J{\"u}rgen Potthoff, and Robert Schrader, \emph{Heat kernels
  on metric graphs and a trace formula}, Adventures in mathematical physics,
  Contemp. Math., vol. 447, Amer. Math. Soc., Providence, RI, 2007,
  pp.~175--198. \MR{2423580 (2010b:81163)}

\bibitem[Kuw10]{Kuw10}
Kazumasa Kuwada, \emph{Duality on gradient estimates and {W}asserstein
  controls}, J. Funct. Anal. \textbf{258} (2010), no.~11, 3758--3774.
  \MR{2606871 (2011d:35109)}

\bibitem[Led93]{ledoux-bourbaki}
Michel Ledoux, \emph{In\'egalit\'es isop\'erim\'etriques en analyse et
  probabilit\'es}, Ast\'erisque (1993), no.~216, Exp.\ No.\ 773, 5, 343--375,
  S{\'e}minaire Bourbaki, Vol. 1992/93. \MR{1246403}

\bibitem[Led94]{Le}
M.~Ledoux, \emph{A simple analytic proof of an inequality by {P}. {B}user},
  Proc. Amer. Math. Soc. \textbf{121} (1994), no.~3, 951--959. \MR{1186991}

\bibitem[Led96]{ledoux-stflour}
Michel Ledoux, \emph{Isoperimetry and {G}aussian analysis}, Lectures on
  probability theory and statistics ({S}aint-{F}lour, 1994), Lecture Notes in
  Math., vol. 1648, Springer, Berlin, 1996, pp.~165--294. \MR{1600888}

\bibitem[Led03]{Le2}
M.~Ledoux, \emph{On improved {S}obolev embedding theorems}, Math. Res. Lett.
  \textbf{10} (2003), no.~5-6, 659--669. \MR{2024723}

\bibitem[Li06]{MR2240167}
Hong-Quan Li, \emph{Estimation optimale du gradient du semi-groupe de la
  chaleur sur le groupe de {H}eisenberg}, J. Funct. Anal. \textbf{236} (2006),
  no.~2, 369--394. \MR{2240167}

\bibitem[LV09]{LV}
John Lott and C{\'e}dric Villani, \emph{Ricci curvature for metric-measure
  spaces via optimal transport}, Ann. of Math. (2) \textbf{169} (2009), no.~3,
  903--991. \MR{2480619}

\bibitem[Pos09]{Pos09}
Olaf Post, \emph{First order approach and index theorems for discrete and
  metric graphs}, Ann. Henri Poincar\'e \textbf{10} (2009), no.~5, 823--866.
  \MR{2533873}

\bibitem[Pos12]{Pos12}
\bysame, \emph{Spectral analysis on graph-like spaces}, Lecture Notes in
  Mathematics, vol. 2039, Springer, Heidelberg, 2012. \MR{2934267}

\bibitem[RS72]{RS72}
Michael Reed and Barry Simon, \emph{Methods of modern mathematical physics.
  {I}. {F}unctional analysis}, Academic Press, New York-London, 1972.
  \MR{0493419}

\bibitem[Shi97]{Shi97}
Ichiro Shigekawa, \emph{{$L^p$} contraction semigroups for vector valued
  functions}, J. Funct. Anal. \textbf{147} (1997), no.~1, 69--108. \MR{1453177
  (98g:60102)}

\bibitem[Shi00]{Shi00}
\bysame, \emph{Semigroup domination on a {R}iemannian manifold with boundary},
  Acta Appl. Math. \textbf{63} (2000), no.~1-3, 385--410, Recent developments
  in infinite-dimensional analysis and quantum probability. \MR{1834233
  (2002f:58065)}

\bibitem[Sto10]{Sto10}
Peter Stollmann, \emph{A dual characterization of length spaces with
  application to {D}irichlet metric spaces}, Studia Math. \textbf{198} (2010),
  no.~3, 221--233. \MR{2650987 (2011i:30052)}

\bibitem[Stu06a]{St1}
Karl-Theodor Sturm, \emph{On the geometry of metric measure spaces. {I}}, Acta
  Math. \textbf{196} (2006), no.~1, 65--131. \MR{2237206}

\bibitem[Stu06b]{St2}
\bysame, \emph{On the geometry of metric measure spaces. {II}}, Acta Math.
  \textbf{196} (2006), no.~1, 133--177. \MR{2237207}

\bibitem[Tay96]{Tay}
Michael~E. Taylor, \emph{Partial differential equations}, Texts in Applied
  Mathematics, vol.~23, Springer-Verlag, New York, 1996, Basic theory.
  \MR{1395147}

\bibitem[Var89]{Varo}
N.~Th. Varopoulos, \emph{Small time {G}aussian estimates of heat diffusion
  kernels. {I}. {T}he semigroup technique}, Bull. Sci. Math. \textbf{113}
  (1989), no.~3, 253--277. \MR{1016211}

\end{thebibliography}
\bibliographystyle{amsalpha}
%
\end{document}